\documentclass[english]{amsart}

\usepackage{amsmath, amsfonts, amssymb, mathrsfs, mathabx, amsthm}
\usepackage[shortlabels]{enumitem}
\usepackage{babel} 
\usepackage[utf8]{inputenc} 
\usepackage[T1]{fontenc}		
\usepackage[text={445pt,630pt}, centering, a4paper,
marginparwidth=2cm]{geometry}
\usepackage{tikz}
\usetikzlibrary{positioning}
\usepackage{graphicx}
\usepackage{tabularx}
\usepackage[hidelinks,pdfusetitle]{hyperref}
\usepackage{bbm}
\usepackage{subcaption}
\usepackage{mathdots}
\usepackage{cancel}
\usepackage{bbold}

\DeclareMathOperator{\card}{card}

\DeclareMathOperator{\vect}{Vect}
\DeclareMathOperator{\im}{Im}
\DeclareMathOperator{\id}{Id}

\DeclareMathOperator{\R}{\mathbb{R}}
\DeclareMathOperator{\C}{\mathbb{C}}

\DeclareMathOperator{\N}{\mathbb{N}}

\DeclareMathOperator{\leb}{Leb}

\DeclareMathOperator{\supp}{supp}

\DeclareMathOperator{\spec}{Sp}
\DeclareMathOperator{\mult}{m}

\renewcommand{\ker}{\mathrm{Ker}\,}

\newcommand{\atom}{\mathfrak{A}}

\newcommand{\anz}{\mathfrak{A}^*}
\newcommand{\ca}{\mathcal{A}}
\newcommand{\cb}{\mathcal{B}}

\newcommand{\cf}{\mathcal{F}}

\newcommand{\ci}{\mathcal{I}}

\newcommand{\cp}{\mathcal{P}}
\newcommand{\cz}{\mathcal{Z}}
\newcommand{\norm}[1]{\left\lVert\,#1\,\right\rVert}
\newcommand{\ind}[1]{\mathbb{1}_{#1}}
\newcommand{\expp}[1]{\mathrm{e}^{#1}}
\newcommand{\rd}{\mathrm{d}}
\newcommand{\muae}{\text{ a.e.}}

\newcommand{\un}{\mathbb{1}}

\newcommand{\priv}[1]{\backslash\{#1\}}

\newcommand{\distatom}{\atom_{\mathrm{dist}}}

\newcommand{\critatom}{{\atom_{\mathrm{crit}}}}
\newcommand{\subA}{\tilde{A}}

\newcommand{\cste}{\mathrm{c}}

\usepackage[draft]{fixme} 
\FXRegisterAuthor{pa}{apa}{PAZ}
\FXRegisterAuthor{jf}{ajf}{JFD}
\FXRegisterAuthor{kl}{akl}{KL}
\fxusetheme{colorsig}

\newcommand{\scal}[1]{\langle #1\rangle}

\theoremstyle{theorem}
\newtheorem{theoI}{Theorem}
\newtheorem{theo}{Theorem}[section]
\newtheorem{prop}[theo]{Proposition}
\newtheorem{defi}[theo]{Definition}
\newtheorem{lemme}[theo]{Lemma}
\newtheorem{lem}[theo]{Lemma}

\newtheorem{cor}[theo]{Corollary}

\theoremstyle{remark} 
\newtheorem{ex}[theo]{Example}
\newtheorem{rqe}[theo]{Remark}
\newtheorem{rem}[theo]{Remark}

\begin{document}
\title{Atoms and associated spectral properties for  positive  operators on $L^p$}

\author{Jean-Fran\c{c}ois Delmas}
\address{Jean-Fran\c{c}ois Delmas,  CERMICS, \'{E}cole des Ponts, France}
\email{jean-francois.delmas@enpc.fr}

\author{Kacem Lefki}
\address{Kacem Lefki, Univ.\ Gustave Eiffel, Univ.\
    Paris Est Creteil, CNRS, F-77454 Marne-la-Vall\'ee, France}
  \email{kacem.lefki@univ-eiffel.fr}

\author{Pierre-Andr\'{e} Zitt}
\address{Pierre-Andr\'{e} Zitt, Univ.\ Gustave Eiffel, Univ.\
    Paris Est Creteil, CNRS, F-77454 Marne-la-Vall\'ee, France}
  \email{Pierre-Andre.Zitt@univ-eiffel.fr}

\date{\today}

\thanks{This work is partially supported by Labex B\'ezout reference ANR-10-LABX-58 and
  SNF grant 200020-196999}



\begin{abstract}
  Inspired by Schwartz, Jang-Lewis and Victory, who study in particular
  generalizations  of triangularizations  of matrices  to operators,  we
  shall  give  for  positive  operators on  Lebesgue  spaces  equivalent
  definitions of atoms (maximal  irreducible sets).  We also characterize
  positive power compact  operators having a unique  non-zero atom which
  appears as a  natural generalization of irreducible  operators and are
  also  considered  in  epidemiological  models.   Using  the  different
  characterizations  of atoms,  we also  provide a  short proof  for the
  representation of the  ascent of a positive power  compact operator as
  the maximal length in the graph of critical atoms.
\end{abstract}

\maketitle

\section{Introduction and main results}
\subsection{Setting and main goals}
We consider the Lebesgue space $L^p=L^p(\Omega, \cf, \mu)$ with
$p\in (1, +\infty )$, and a state space $\Omega$ endowed with a
$\sigma$-field $\cf$ and a non-zero $\sigma$-finite measure $\mu$.
Let $T$ be a positive bounded operator on $L^p$.
 For  $A\in  \cf$,  we  denote  by  $T(A)$  the  support  of
  $T(\ind{A})$ (if  $\ind{A}$ does  not belong to  $L^p$, then  one can
  replace it by  $f \ind{A}$ for any positive function  $f\in L^p$) which
  is defined up to sets of $\mu$-zero measure. Then, we say a set $A\in \cf$
  is \emph{invariant} if $T(A) \subset A$.
A set $A$ is \emph{co-invariant} if $A^c$ is
invariant (or equivalently if $A$ is invariant for the dual operator
$T^\star$).  The collection of \emph{admissible} sets corresponds to
the $\sigma$-field $\ca\subset \cf$ generated by the invariant sets.
We define the \emph{atoms} as the minimal admissible sets with
positive measure.  An atom is \emph{non-zero} if $T$ restricted to
this atom is non-zero. An atom is \emph{critical} if it is non-zero
and the spectral radii of $T$ and of $T$ restricted to this atom are
equal.

\medskip

Building on works by Schwartz \cite{schwartz_61} and Jang-Lewis and Victory
\cite{janglewisvictory}, 
that study in particular
generalizations of triangularizations of matrices to operators, 
our aim in this work is threefold:
\begin{enumerate}
\item give several equivalent definitions of atoms,
\item describe all the nonnegative eigenfunctions of $T$ using distinguished atoms,
  allowing a characterization of  operators $T$ having  a unique non-zero atom;
    \item describe all the generalized eigenfunctions of $T$ whose
      eigenvalue is the spectral radius of $T$,
      and represent the ascent of $T$ as the maximal length in the  graph of
    critical  atoms.
\end{enumerate}  
Except the characterization of atoms, all our results are proved under
the assumption that $T$ is power compact. 

We now give details on each of these aspects, discussing the relevant literature
after each statement.

\subsection{On atoms}

For a  measurable set $A$,  we consider its \emph{future}  $F(A)$ (resp.
its \emph{past}  $P(A)$) as the smallest  invariant (resp. co-invariant)
set containing $A$.   When $T$ is seen as the  transmission operator for
an epidemic  propagation, see Delmas, Dronnier  and Zitt \cite{ddz-sis},
the future $F(A)$  can be interpreted as the  sub-population of $\Omega$
which might be  infected by an epidemic starting in~$A$,  and $P(A)$ can
be interpreted as  the sub-population of $\Omega$  which may contaminate
the population $A$. Motivated by the point of view of successive
infections, we prove  the following interpretation of the future in 
   Corollary~\ref{cor:expT}, for $A\in \cf$:
   \[
    \expp{T}(A)= \bigcup_{n \in \N} T^n(A)=F(A).
  \]

We say the   operator    $T$   on    $L^p$   is
\emph{irreducible}  if  its  only  invariant  sets  are  a.e.\  equal  to
$\emptyset$  or  $\Omega$; in particular
$F(A) = P(A) = \Omega$ for any measurable set $A$ with positive measure.
We  say   that  a  set  $A\in  \cf$  is
\emph{irreducible}  if  it has  positive measure and 
  the  operator $T$  restricted to  the set  $A$ is
irreducible. 

Motivated by the example of Volterra operator, see Example~\ref{ex:cvx_volterra}
below for details, 
and by an analogy with
order theory, we say that an admissible set $A$ is
\emph{convex} if $A=P(A) \cap F(A)$.

Our  first result  gives equivalent
 characterizations of atoms using convex sets and irreducible sets.

\begin{theoI}[Equivalent definitions of atoms]\label{thI:equi_atomes}
  Let $T$ be a positive operator on  $L^p$ with $p \in (1, +\infty)$. The  following
  properties are equivalent.

\begin{enumerate}[(i)]
\item \label{thI:item:atom}
  The set $A$ is an atom. 
\item \label{thI:item:min_interval}
  The  set $A$ is
      minimal among convex sets with positive measure.
  \item \label{thI:item:max_irre2}
  The set $A$ is an admissible
  irreducible set.
\item \label{th:item:max_irre}
  The set $A$ is a maximal irreducible
  set.
\end{enumerate}
\end{theoI}

Following 
\cite{schwartz_61}, we also give an at most countable partition of $\Omega$ in
atoms and a (possibly empty) set  $\Omega_0$ such that $T$ restricted to
each atom is irreducible; 
if  the  operator  $T$  is power  compact,  then  $T$ 
is     quasi-nilpotent   on $\Omega_0$.

\begin{rem}[Related notions and results.]
  Various  definitions and  properties of  atoms already  appear in  the
  literature.  Our definition  of invariance and atoms  are adapted from
  Schwartz~\cite{schwartz_61},                  see                 also
  Victory~\cite{victory82,victory82eigenspace}.   The  past   of  a  set
  appears in  Nelson~\cite{nelson74} (as the closure)  and in Jang-Lewis
  and Victory~\cite{janglewisvictory} (as closure  for bands in a Banach
  lattice).    Irreducibility    corresponds   to   ideal-irreducibility
  from Schaefer~\cite{schaefer_74}.   Maximal irreducible  sets appear  in
  \cite{nelson74} and \cite{victory82} for  kernel operators (where they
  are called  components), and Omladič and  Omladič \cite{omladic95} for
  more  general  Banach  lattices   (where  they  are  called  classes).
  Convexity   of   atoms  is   used   in   the  proof   of   \cite[Lemma
  12]{schwartz_61};  the   irreducible  bands  used  in   the  Frobenius
  decomposition   from  Jang   and  Victory   \cite{jlv90}  are   convex
  irreducible sets, and the  semi-invariant bands, considered by Bernik,
  Marcoux   and  Radjavi~\cite{bernik12}   are  in   particular  convex.
  However, to the best of our  knowledge, convexity has not been studied
  for its  own sake  in this  setting, and  the equivalence  provided by
  Theorem~\ref{thI:equi_atomes} is new.

Finally, the decomposition of the space in atoms and a part where
$T$ is quasi-nilpotent is essentially due to Schwartz \cite{schwartz_61}.
It corresponds,
 for nonnegative matrices,  to the Frobenius normal
form introduced by Victory \cite{victory85}, that is, a block
triangularization of the matrix according to the communication
classes.
 Notice that the triangularization of matrices has been
extended to (bounded) operators in Banach spaces by Ringrose
\cite{ringrose62} using invariant spaces, see also Dowson
\cite[Section 2]{dowson}.
\end{rem}

\subsection{Nonnegative eigenfunctions}

From now on we assume that the  positive 
 operator $T$ is power compact with
positive spectral radius $\rho(T)>0$.
  For a (non-zero)  eigenfunction
$v$ of $T$, let  $\rho(v)$ denote the corresponding eigenvalue: $Tv=
\rho(v)\, v$ (and similarly for left eigenfunctions).

Let us recall briefly
two key results on nonnegative eigenfunctions for positive power compact operators, see Theorem~\ref{theo:rappel}.
 Let
$\mult(\lambda, T)$ denote the algebraic multiplicity of
$\lambda\in \C^*$, that is, the dimension of
$ \bigcup_{k \in \N} \ker (T - \lambda \id)^k$.  Recall that
$\lambda\in \C^*$ is a simple eigenvalue if $\mult(\lambda, T)=1$.
According to Krein-Rutman
theorem, $\rho(T)$ is an eigenvalue of $T$, and there exists
corresponding nonnegative right and left eigenfunctions.
Furthermore, if $\rho(T)$ is positive and if $T$ is irreducible,
the Perron-Jentzsch theorem states that the eigenvalue $\rho(T)$ is simple, and the
corresponding right and left eigenfunctions are in fact positive.

Our first result characterizes \emph{monatomic} operators, that is,
operators having a unique non-zero atom,
in terms of nonnegative eigenfunctions.

\begin{theoI}[Characterization of monatomic operators]
  \label{thI:carac_monat}
  Let $T$  be a positive power  compact operator on $L^p$ with $p \in (1, +\infty)$ with positive
  spectral radius.  The following properties are equivalent.

\begin{enumerate}[(i)]
\item \label{thI:item:T_mono_at}
  The operator $T$ is a monatomic.
\item \label{thI:item:ex_vp_gd_simple}
  There exist a unique right and a unique left
  nonnegative eigenfunctions of  $T$ related to a  non-zero eigenvalue,
  and $\rho(T)$ is a simple eigenvalue of $T$.
\item \label{thI:item:ex_vp_gd_supp_non_vide}
  There exist a unique right and a unique left nonnegative eigenfunctions of $T$ related to a 
  non-zero   eigenvalue, say $u$ and $v$,   and
  $\supp       (u)       \cap        \supp(v)    $ has positive measure.
\end{enumerate}
  Furthermore, when the operator
  $T$ is monatomic,  if  $u$ and
  $v$  denote its  unique right and  left nonnegative eigenfunctions, then
  $\rho(u)=\rho(v)=\rho(T)$ and  $\supp (u) \cap \supp(v) $ is the non-zero
  atom.
\end{theoI}

\begin{rem}[On monatomicity]
Monatomicity  is a natural extension of  irreducibility, and 
generalizes the notion of quasi-irreducibility defined for symmetrical
operators, see Bollob\'{a}s, Janson and Riordan \cite[Definition 2.11]{bollobas07}.  
Monatomic operators naturally appear when
studying the concavity property of the function
$\eta \mapsto \rho(T M_\eta)$ where $\eta $ is a $[0,1]$-valued
measurable function defined on $\Omega$ and $M_\eta$ the
multiplication by $\eta$ operator defined on $L^p$, see
for example Delmas, Dronnier and Zitt~\cite[Lemma~7.3]{ddz-Re} and its discussion
for additional references in particular in epidemiology.
\end{rem}

More generally, we may  characterize nonnegative eigenfunctions in terms
of  the  atoms appearing  in  their  support. Let  us  give  a few  more
     definitions.
  Let $\atom$ denote  the set
of atoms  (which  is  at most  countable and might be  empty); and we
introduce a (partial) order $\preccurlyeq$ and the corresponding strict
order $\prec$ on this set
   (see Proposition~\ref{prop:order}): for 
  two   atoms  $A,  B$,  we  write  $B   \prec  A$  if
$B\subset F(A)\setminus A$.  A family of atoms is an \emph{antichain} if
no two atoms  in the family satisfy  $B\prec A$.  For any  atom $A$, let
$\rho(A)$ be the spectral radius of the  restriction of $T$ to $A$.  Let us
say  that an  atom $A$  is \emph{distinguished}  if, for  any atom  $B$,
$B\prec  A$ implies  that $\rho(B) < \rho(A) $,  and that  an eigenvalue
$\lambda$  is  \emph{distinguished}  if there  exists  a
distinguished atom $A$ with $\rho(A) = \lambda$. 
For $\lambda\in  \R_+^*$, we  consider the
(finite but possible empty) set $ \atom(\lambda)$ of atoms with spectral
radius $\lambda$  and the  subset $\distatom(\lambda)$  of distinguished
atoms associated to $\lambda$:
\[
  \atom(\lambda)=\left\{ A \in \atom \, \colon\,
  \rho(A)=\lambda
\right\}
\quad\text{and}\quad
\distatom(\lambda)=\left\{
  A \in \atom(\lambda) \, \colon\,
  \text{$A$ is distinguished}
\right\}.
\]
To  any distinguished  atom $A$,  we  may associate  a unique  (up to  a
multiplicative  constant) nonnegative  eigenfunction denoted  $w_A$ such
that  $\supp  (w_A) =  F(A)$  and  furthermore $\rho(w_A)=\rho(A)$  (see
Lemma~\ref{prop:nonneg_eigen_dist_atom}~\ref{prop:item:ex_vp_dist}); and
then the following holds.
\begin{theoI}[Characterization of nonnegative right
  eigenfunctions]\label{thI:carac_nonneg_vp}
Let $T$ be a positive power compact operator on $L^p$ with $p \in (1,
+\infty)$. Let $\lambda > 0$.
We have the following properties. 
  \begin{enumerate}[(i)]
  \item\label{itemI:l-dist-atom}
    There exists a nonnegative eigenfunction of $T$
   with eigenvalue
  $\lambda$ if and only if $\lambda$ is a distinguished eigenvalue. 
  \item\label{item:antichain_dist}
The set
$\distatom(\lambda)$ is a (possibly empty)  finite antichain of atoms, and the
family $(w_A)_{A\in \distatom(\lambda)}$ is linearly
    independent.
 \item\label{itemI:V+=conical_hull}
  The cone of nonnegative right eigenfunctions of $T$ with
    eigenvalue $\lambda$
    is the conical hull of  $\{w_A\, \colon\, A\in
    \distatom(\lambda)\}$, and more precisely:
    if $v$ is a nonnegative eigenfunction with
  $\rho(v)=\lambda$, then we have:
    \[
      v = \sum_{A\in \distatom(\lambda)} \cste_A w_A,
    \]
    where $\cste_A \in \R_+$, and $\cste_A >0$ if and only if $A\subset \supp(v)$. 
  \end{enumerate}
\end{theoI}

\begin{rem}[Related results]
  The theorem is in essence a reformulation of results by Jang-Lewis
  and Victory.  More precisely, definitions and characterization of
  distinguished atoms and eigenvalues appear
  in~\cite{jang00,janglewisvictory,tam01,victory82} ;
  Point~\ref{itemI:l-dist-atom} is in \cite[Theorem
  IV.1]{janglewisvictory} in the more general context of power compact
  operators on a Banach lattice with an order continuous norm, and
  Point~\ref{itemI:V+=conical_hull} appears in \cite[Corollary
  1]{victory82} for power compact kernel operators on $L^p$.

  The salient point of our approach is that it leverages the decomposition
  of the multiplicities of the eigenvalues given in
  \cite[Theorem~7]{schwartz_61} and our characterization of
  atoms from Theorem~\ref{thI:equi_atomes} to provide simpler and shorter proofs. 
\end{rem}

\subsection{Critical atoms and generalized eigenspace}

We  now  give  a  particular   attention  to  the  atoms  associated  to
$\rho(T)$.
We define the generalized eigenspace:
\[
  K(T) = \bigcup\limits_{k \in \N} \ker \left( T - \rho(T)\id
  \right)^k .
\]
Following \cite{dunford88} and \cite{konig86}, we
define, with the convention $\inf \emptyset = +\infty$, the
\emph{ascent}   of $T$ at its
spectral radius $\rho(T)$ by:
\[
   \alpha_T  = \inf  \{ k  \in \N :  \ker (T - \rho(T)  \id)^k
   =\ker (T - \rho(T) \id)^{k+1}\} .
\]
It is well-known, see \cite{konig86}, that when the operator $T$ is
power compact, the ascent  $\alpha_T$ is finite.

We   say  that   an   atom   $A$  is   \emph{critical}   when  we   have
$\rho(A) = \rho(T)$, and we denote $ \critatom = \atom(\rho(T))$ the set
of the critical atoms.  For $n\geq 1$, a \emph{chain of length $n$} is a
sequence  $(A_0, \dots,  A_n)$ of  elements  of $  \critatom$ such  that
$ A_{i+1}\prec A_i$ for all $0\leq  i < n$.  The \emph{height} $h(A)$ of
a critical atom $A$ is one plus  the maximum length of a chain starting
at $A$.

Our last result is the following. 
\begin{theoI}[Ascent and maximal height]
  \label{thI:gen_eigenspace}
  Let  $T$  be   a  positive  power  compact  operator   on  $L^p$  with
  $p \in (1, +\infty)$ with positive  spectral radius. 
  Then the ascent of $T$ at its spectral radius $\rho(T)$ is equal to the
  maximal height of the critical atoms: 
\[
  \alpha_T = \max_{A \in \critatom} h(A).
\]
\end{theoI}
This result is also stated in
\cite{jlv90,janglewisvictory} for positive power compact operators on
Banach lattices with order continuous norm. Here also, we  provide a
shorter proof using properties of convex sets.

\subsection{Structure of the paper}

After recalling basic notions on Banach spaces in
Section~\ref{sec:notation},
we introduce the invariant/admissible sets and the atoms in
Section~\ref{sec:atom}, then we define the future and the past of a set
in Section~\ref{sec:future-past}, the irreducible sets in
Section~\ref{sec:irreduc},
the convex sets in Section~\ref{sec:cvxe} and the order $\preccurlyeq $
 in Section~\ref{sec:order}. 
We then study properties and characterizations of atoms in
Sections~\ref{sec:prop-atom} and~\ref{chara-atom}, and we stress some relation between the
atoms  of $T$ and $T^n$ in
Section~\ref{sec:admin-Tn}.

To build intuition, we devote Section~\ref{subsec:countable}
to the  particular case where  $\Omega$ is countable, and therefore a  union of atoms.

\medskip
We characterize the cone of nonnegative eigenfunctions with the same eigenvalue for power compact positive operators in
Section~\ref{sec:nonneg_eigen} and prove Theorem~\ref{thI:carac_nonneg_vp}  (see
Theorem~\ref{th:carac_nonneg_vp_v2}). 
Section~\ref{sec:mono} is devoted to the proof of
Theorem~\ref{thI:carac_monat} (see Theorem~\ref{th:carac_monat}) on the
characterization of monatomic operators.

\medskip
Section~\ref{sec:gen_eigenspace} is devoted to the generalized eigenfunction
associated to the eigenvalue $\rho(T)$ 
and Theorem~\ref{thI:gen_eigenspace}
(see Theorem~\ref{th:base_K} and Corollary~\ref{cor:ascent}).

\section{Notations}
 \label{sec:notation}

\subsection{Ordered set}

Let  $(E, \leq  )$  be a  (partially) ordered  set,  also called  poset.
Whenever  it exists,  the \emph{supremum}  of $A\subset  E$, denoted  by
$\sup(A)$, is the  least upper bound of $A$ (formally,  $\sup(A)\in E$ is
defined by: for all $x\in A$, $x  \leq \sup(A)$ and if for some $z\in E$
one  has  $x\leq  z$ for  all  $x\in  A$,  then  $\sup(A) \leq  z$).   A
collection  $(x_i)_{i  \in  \mathcal{I}}$  of  elements  of  $E$  is  an
\emph{antichain}  if  for  all  distinct $i,  j  \in  \mathcal{I}$,  the
elements $x_i$ and  $x_j$ are not comparable for the  order relation.  A
set $D \subset E$ is a \emph{downset} if for all $x \in D, y \in E$, the
relation $y \leq x$ implies $y \in D$.

\subsection{Banach space and Banach lattice}
\label{sec:banach}
Let $(X, \norm{\cdot})$ be a complex Banach space not reduced to~$\{0\}$. An operator $T$ on $X$ is a bounded linear (and thus
continuous) map from $X$ to itself. Its operator norm is given by:
\[
  \norm{T}_X= \sup\left\{ \norm{T(x)}\, \colon\, x \in X \text{ s.t. }
  \norm{x} = 1\right\},
\]
its spectrum by  $\spec(T) = \{\lambda \in \C\, \colon\, T - \lambda \id
\text{ has no  inverse}\}$, where $\id$ is the identity operator on
$X$, and its spectral radius (see \cite[Theorem~18.9]{rudin87}) by:
\begin{equation}
   \label{eq:def-rho}
  \rho(T) = \sup\left\{ |\lambda| \,\colon\, \lambda \in \spec
    (T)\right\} =\lim_{n\rightarrow \infty } \norm{T^n}_X^{1/n}.
\end{equation}
By convention we set $T^0=\id$.

Let $X^\star$ denote the (topological) dual Banach space of $X$, that is the
set of all the bounded linear forms on $X$.  For $x\in X$, $x^\star\in X^\star$, let
$\langle  x^\star, x  \rangle$  denote the duality  product. For  an
operator  $T$,  the   dual  operator  $T^\star$  on  $X^\star$   is  defined  by
$\langle T^\star x^\star,  x \rangle=\langle x^\star,  Tx \rangle$ for all  $x\in X$,
$x^\star\in X^\star$.

If $\lambda  \in \C$ and  $v\in X\priv{0}$ satisfy $T(v)=\lambda  v$, we
say  that $v$  is a  right eigenfunction  of $T$,  $\lambda$ is  a right
eigenvalue  of $T$,
  and, in  view of  the
  forthcoming    Corollary~\ref{cor:atoms_in_support},    shall    write
  $\lambda=\rho(v)$.
Any  right  eigenvalue
(resp.   eigenfunction)   of  $T^\star$   is   called   a  left   eigenvalue
(resp. eigenfunction) of $T$.
Unless there is an ambiguity, we shall simply write \emph{eigenvalue} and
\emph{eigenfunction} for  right  eigenvalues and  eigenfunctions.

\medskip

An ordered real Banach space $(X, \norm{\cdot}, \leq )$ is a real Banach
space $(X,  \norm{\cdot})$  with an  order relation
$\leq$.   For any  $x \in  X$,  we define  $|x| =  \sup(\{x, -x\})$  the
supremum   of   $x$   and    $-x$   whenever   it   exists.    Following
\cite[Section~2]{schaefer_74},     the      ordered     Banach     space
$(X, \norm{\cdot}, \leq)$ is a \emph{Banach lattice} if:
\begin{enumerate}
\item For  any $x, y, z  \in X, \lambda \geq  0$ such that $x  \leq y$, we
  have $x + z \leq y + z$ and $\lambda x \leq \lambda y$.
\item For any $x, y \in X$, there exists  a supremum of $x$ and $y$ in $X$.
\item  For  any  $x,  y  \in  X$   so  that  $|x|  \leq  |y|$,  we  have
  $\norm{x} \leq \norm{y}$.
\end{enumerate}

Let $(X, \norm{\cdot}, \leq)$ be  a real Banach lattice.
A vector subspace $Y$ of  $X$ is an \emph{ideal} if:
\[
  x\in Y,\,  y \in X, \, |y| \leq |x| \quad \Rightarrow \quad y \in Y.
\]
Let $T$ be an operator on $X$. A set $Z\subset X$ is $T$-invariant or
simply  \emph{invariant} when there is no ambiguity, if $T(Z)\subset
Z$. An operator $T$ on $X$  is \emph{positive} if the positive cone $X_+ = \{x
\in X \,\colon\, x \geq 0 \}$ is invariant. 
The  operator $T$  is  \emph{ideal-irreducible} if the only invariant
closed ideals of $X$   are $\{0\}$ and $X$, see \cite[Definition~8.1]{schaefer_74}.

\medskip

Any Banach  lattice $X$  and any  operator $T$ on  $X$ admits  a natural
complex  extension.  The  spectrum  of  $T$ will  be  identified as  the
spectrum of its complex extension and denoted by $\spec(T)$, furthermore
by \cite[Lemma  6.22]{abramovich02}, the spectral radius  of the complex
extension             is            also             given            by
$\lim_{n\rightarrow \infty } \norm{T^n}_X^{1/n}$. 
Moreover, by \cite[Corollary 3.23]{abramovich02}, if $T$ is positive
(seen as an operator on the real Banach lattice $X$), then $T$ and its complex
extension have the same norm. 
If  $S$ and $T$ are  two operators on $X$,
we write  $T \leq S$ if the operator $S - T$ is positive.
If the operators $T, S$ and $S-T$ are positive, then we have
$\rho(T) \leq \rho(S)$, see \cite[Theorem 4.2]{marek70}.

\subsection{Lebesgue spaces}
\label{sec:leb}
Let  $(\Omega, \mathcal{F}, \mu)$ be a measured space with $\mu$ a
$\sigma$-finite measure such that $\mu(\Omega) > 0$. For any
$\mathcal{A} \subset \mathcal{F}$, we denote by $\sigma(\mathcal{A})$
the $\sigma$-field generated by $\mathcal{A}$.
If $f, g$ are two real-valued measurable functions defined on $\Omega$,
we write  $f \leq g$ a.e.\ (resp.  $f = g$ a.e.)
when $\mu(\{f > g\})  = 0$ (resp. $\mu(\{f \neq g \})  = 0$), and denote
$\supp(f )= \{f \neq 0\}$ the support of $f$.
We say that a real-valued measurable function $f$ is nonnegative when $f \geq 0$ a.e., and we say that $f$ is positive, denoted $f > 0$ a.e., when $\mu(\{f \leq 0\}) = 0$.
If $A, B \subset \Omega$ are measurable sets, we write $A \subset B$
a.e.\ (resp $A = B$ a.e.) when  $\un_A \leq \un_B$ a.e.\ (resp.  $\un_A
= \un_B$ a.e.). For the sake of clarity, we will omit to write a.e.\ in the  proofs. 
We shall consider the following definition of minimal/maximal sets.

\begin{defi}[Minimal or maximal set for a property $\cp$]
  \label{def:min-max}
  Let  $\cp \subset  \cf$ be  a class  of measurable  sets. We  say that
  $A\in \cf$  is \emph{minimal}  for $\mathcal{P}$  if $A\in  \cp$ and  for any
  $B \in \cp$ such that $B \subset A$ a.e., we have $B = \emptyset$ a.e.\
  or $B = A $ a.e.. We say that $A \in \cf$ is \emph{maximal} for $\mathcal{P}$
  if $A^c$ is minimal for $\{B^c : B \in \mathcal{P}\}$.
\end{defi}

We  will  usually   say  ``minimal  +  property  set''   for  a  minimal
(measurable) set for the corresponding property. For example, an atom of
the    measure    $\mu$    is    any   minimal    measurable set with
positive measure,  that is, any  minimal set of $\cp=\{A\in \cf\,
\colon\, \mu(A)>0\}$. 

\begin{lem}[Existence of a minimal set]
   \label{lem:min-exist}
 Let  $\cp \subset  \cf$ be  a class  of measurable  sets stable by
 countable intersection. Then there exists a measurable set  minimal  for $\cp$. 
\end{lem}

\begin{proof}
   We   recall, see
  \cite[Appendix A.5]{essinf} (where the result is stated for  $\mu$  a
  probability measure, but can be easily extended to a $\sigma$-finite
  measure),   that   if
  $\{f_i\, \colon\,  i  \in I \}$ is a (possibly  uncountable) family of
  $[-\infty ,  +\infty ]$-valued  measurable defined on  $\Omega$, then there
  exists a $[-\infty , +\infty ]$-valued measurable function $f$, called
  the essential infimum of $\{f_i : i \in I \}$ such that:
\begin{enumerate}[(i)]
\item For all $i\in I$,  $f_i \geq f$ a.e..
\item If $g$ is another $[-\infty , +\infty ]$-valued measurable
  function satisfying (i), then a.e.\ $f\geq g$. 
\end{enumerate}
Furthermore, there exists an at most countable set $I'\subset I$ such
that a.e.\ $f= \inf _{i\in I'} f_i$. 
\medskip

We      consider       $f$      the      essential       infimum      of
$\{\ind{B}\,  \colon\,  B\in \cp\}$.   Thus,  there  exists an  at  most
countable     set     $\cp'\subset      \cp$     such     that     a.e.\
$f=\inf_{B\in   \cp'}  \ind{B}$,   that  is   a.e.\  $f=\ind{B'}$   with
$B'=\bigcap_{B\in  \cp'}  B$.   Since   $\cp$  is  stable  by  countable
intersection, we get  that $B'$ belongs to $\cp$. Property  (i) above on
the  essential infimum  implies also  that $B'\subset  B$ a.e.\  for all
$B\in \cp$.   Thus the set $B'$  is minimal for $\cp$.  This provides the
existence of a minimal set for $\cp$.
\end{proof}

\medskip

For       a      measurable       function      $f$,       we      write
$\mu(f)=\int f\, \rd \mu=\int_\Omega f(x)\,  \mu(\rd x)$ the integral of
$f$   with  respect   to   $\mu$   when  it   is   well  defined.    For
$p \in (1, +\infty)$, the Lebesgue space $L^p(\Omega, \mathcal{F}, \mu)$
is  the set  of  all  real-valued measurable  functions  $f$ defined  on
$\Omega$ whose  $L^p$-norm, $\norm{f}_p  = \mu(|f|^p)^{1/p}$,  is finite
and where functions which are a.e.\ equal are identified.  When there is
no  ambiguity we  shall simply  write $L^p(\Omega)$  or $L^p$.   The set
$(L^p,    \norm{\cdot}_p)$    is    a    Banach    space    with    dual
$(L^q, \norm{\cdot}_q)$, where $1/p+1/q=1$.  The duality product is thus
$\langle  g,  f \rangle  =  \int  fg\, \rd  \mu$  for  $f \in  L^p$  and
$  g \in  L^q$.  The  Banach space  $L^p$ endowed  with the  usual order
$f\leq g$, that is $\mu(\{f>g\})=0$, is a Banach lattice.  The positive cone
$L^p_+$ is the  subset of $L^p$ of nonnegative  functions.  According to
\cite[Section~2]{vovo2016}  and  \cite[Theorem 5.14, p.94]{schaefer_74}, its closed ideal are the sets:
\begin{equation}
   \label{eq:def-LpA}
L^p_A=  \left\{ f\in L^p\, \colon\, f \ind{A^c}=0 \right\},
\end{equation}
where $A\subset \Omega$ is measurable.

\medskip

Let $T$ be an operator on $L^p$. Thanks  to
\cite[Corollary~1.3]{figiel84},   $T$   and
its    complex
extension  on the natural complex extension of $L^p$ have  the  same  $L^p$-norm. 
Let $A \subset  \Omega$ be  measurable.  We define
the restriction operator of  $T$  on
$A$, denoted $T_A$,  by:
\begin{equation}
   \label{eq:def-TA}
   T_A=M_A \, T\,  M_A,
   \quad\text{where the operator $ M_A$ is the multiplication by  $\un_A$},
\end{equation}
and,  if   $\mu(A)>0$, we denote  by  $T|_A$  the  corresponding
operator on  $L^p(A)$, where the  set $A$ is  endowed with the  trace of
$\cf$ on  $A$ and the  measure $\mu|_A (\cdot)=\mu(A\cap  \cdot)$.  When
there is no ambiguity on the operator $T$, we simply write $\rho(A)$ for
the spectral  radius of $T_A$ (and  of $T|_A$).  In particular,  we have
$\rho(\Omega)=\rho(T)$  and $\rho(A)  =  0$  if $\mu(A)  =  0$.  If  the
operator $T$ is positive, we also have that:
\[
  A \subset B \quad \Longrightarrow \quad \rho(A) \leq  \rho(B).
\]

\medskip

A kernel $k$ is a measurable nonnegative function defined on
$(\Omega^2, \cf^{\otimes 2})$. When possible, we define  for a real-valued measurable
function $f$ defined on $\Omega$ the function $T_k(f)$ by:
\begin{equation}
   \label{eq:def-Tk}
  T_k(f) (x) = \int_\Omega k(x,y) f(y)\,  \mu(\rd y)
\quad\text{for}\quad
x\in \Omega.
\end{equation}
When it is well defined as an operator on $L^p$, we call $T_k$ the
kernel operator associated to $k$.

\section{Atomic decomposition of a positive operator}
\label{sec:atom}
We   consider   the  Lebesgue   space   $L^p=L^p(\Omega, \cf, \mu)$  with $\mu$ a non-zero  $\sigma$-finite measure and
$p \in  (1,+\infty)$.
In this section, we introduce the notion of invariant set, in order to
provide different characterizations of the atoms of a positive bounded operator on $L^p$.

\subsection{Invariance and atoms}
\label{sec:invariance}

  The  ideal-irreducibility of  an operator  can be
described  in terms of sets rather than functions.  We follow  the presentation  of
Schwartz \cite{schwartz_61} (notice $\mu$ is assumed to be finite therein). 

Let $T$ be a positive operator on $L^p$. Let $f \in L^p$ and $g \in L^q$
be two positive functions (with $1/p+1/q=1$). We define the nonnegative function
$k^{[g,f]}_T$ on $\cf^2$ as, for $A, B\in \cf$:
\[
  k^{[g,f]}_T(B, A) = \langle g \un_B,\, T(f \un_A) \rangle=\int_B g(x)
  \, T(f\ind{A})(x)\, \mu(\rd x) .
\]
Notice that:
\begin{equation}
   \label{eq:k-adj}
  k^{[f,g]}_{T^\star} = k^{[g,f]}_T.
\end{equation}

We shall consider the zeros of  $k^{[g,f]}_T$, that is the set:
\begin{equation}
   \label{eq:def-ZT}
 \cz_T=\{(B, A) \in \mathcal{F}^2 : k^{[g,f]}_T(B, A) = 0\}.
 \end{equation}
Let us stress that  the set $\cz_T$ does not  depend on the choice of the
positive functions $f\in L^p$ and $g\in  L^q$; this is indeed a direct
consequence of the following equivalences: 
\begin{equation}
   \label{eq:kT=0}
  k^{[g,f]}_T(B, A) = 0\, \Longleftrightarrow\, \un_B T(f \un_A) = 0 \muae
 \, \Longleftrightarrow\,  T^\star(g \un_B) \un_A = 0 \muae.
 \end{equation}
 For this reason, as long
as we consider  the zeros of $k^{[g,f]}_T$, when there  is no ambiguity,
we  shall simply  write:
\begin{equation}
  \label{eq:def-kT}
  k_T=k^{[g,f]}_T.
\end{equation}
Notice that  for any
$A \in  \mathcal{F}$, the  maps $k_T(\cdot, A)$  and $k_T(A,  \cdot)$ on
$\cf$ are $\sigma$-additive and nonnegative.
We can now introduce the definition of invariant set.

\begin{defi}[Invariant and co-invariant sets]\label{defi:inv}
  Let $T$  be a positive operator  on $L^p$ with $p  \in (1,+\infty)$. A
  set  $A$  is  $T$-invariant  or   simply  \emph{invariant}  if  it  is
  measurable  and $k(A^c,  A) =  0$;  it is  $T$-co-invariant or  simply
  \emph{co-invariant}  if   $A^c$  is   $T$-invariant.   We   denote  by
  $\mathcal{I}$ the class of the invariant sets.
\end{defi}

If $A$ is an invariant set and $B = A \muae$, then $B$ also is invariant.
Note  also  that  $A$  is  $T$-co-invariant   if  and  only  if  $A$  is
$T^\star$-invariant thanks to~\eqref{eq:k-adj}, and that the following
equivalences hold:
\begin{equation}
   \label{eq:Tinv_existsh}
   A \text{ is invariant} \quad \Longleftrightarrow \quad
     \exists h \in
L^p_+,\,\,   \supp(h) = A, \,  \text{ and }\, \, 
 T(h) = 0 \,\text{ on } \, A^c.
 \end{equation}
The next lemma is a direct consequence of the $\sigma$-additivity of
$k_T$. 

\begin{lem}[Countable union and intersection of invariant sets]
  \label{lem:union_int_inv}
Any at most countable union or intersection of invariant (resp. co-invariant) sets is invariant (resp. co-invariant).
\end{lem}
We have the following characterization of invariance using closed ideals.

\begin{lem}[Invariant sets and invariant closed ideals]
  \label{lem:inv_sets_ideals}
Let $T$ be a positive operator on $L^p$ with $p \in  (1,+\infty)$, and $A$ a measurable set. The set   $A$ is $T$-invariant  if and only if the
 closed ideal $L^p_A$ is $T$-invariant.
\end{lem}

\begin{proof}
We first assume that $A$ is invariant. Let $h \in L^p_A$, that is $h\in
L^p$ and $h\un_{A^c} = 0$. Let $f' \in L^p$ and $g \in L^q$ be positive
and set $f=f'+ |h|$. Since $A$ is invariant,  we have $k^{[g,f]}_T(A^c, A) = 0$.
This gives:
\[
  0 = \langle g \un_{A^c}, T(f \un_A) \rangle \geq \langle g \un_{A^c},
  T(|h|) \rangle \geq \langle g \un_{A^c}, |T(h)| \rangle \geq 0,
\]
where we  used the positivity of  $T$ for the inequalities.  We get that
$T(h)\ind{A^c} = 0$,  that is, $T(h) \in L^p_A$. Thus  the ideal $L^p_A$
is invariant.

\medskip

We now assume that the ideal $L^p_A$ is invariant. For $f \in L^p$ and
$ g \in L^q$ positive, we have that $g \un_{A^c} T(f \un_A) = 0$.
Therefore $k^{[g,f]}_T(A^c, A) = 0$, thus the set $A$ is invariant. This
ends the proof.
\end{proof}

\begin{ex}[The Volterra operator]\label{ex:kernel}
  We           consider           the           measured           space
  $(\Omega=[0,  1], \cf=\cb([0,  1]),  \leb)$, with  $\cf$ the  Borel
  subsets of $[0, 1]$ and $\leb$  the Lebesgue measure on $[0, 1]$,
  and the kernel $k$ on $[0, 1]$ defined by:
  \[
    k(x,y) = \un_{\{x \geq  y\}}
    \quad\text{for $x, y \in  [0,1]$}.
  \]
The    corresponding     kernel    operator     $T_k$    given
  by~\eqref{eq:def-Tk}  is  the  so-called 
  Volterra  operator   (see  \cite{barnes90}   for  some   spectral  and
  compactness properties  of Volterra  operators).  One  can see  that a
  measurable    set    $A     \subset    [0,1]$    is    $T_k$-invariant
  (resp.   $T_k$-co-invariant)   if   and  only   if   $A=[a,1]$   a.e.\
  (resp. $A=[0, a]$ a.e.) with $a\in [0, 1]$. 
\end{ex}

We give an immediate application of Lemma~\ref{lem:inv_sets_ideals}.
\begin{lem}[$T$ and $T^n$ invariant sets]
  \label{lem:inv_T^n}
   Let $T$ be  a positive operator on $L^p$ with  $p \in (1,+\infty)$ and
  $n \in \N^*$. Any $T$-invariant set is  $T^n$-invariant. 
\end{lem}

We give another example of invariant sets, that will be useful later on.

\begin{lem}[The support of a nonnegative eigenfunction is invariant]
  \label{lem:supp_inv}
  Let $T$  be a positive  operator on $L^p$ with $p \in  (1,+\infty)$
 and  $v$ be  a nonnegative 
  right eigenfunction  of  $T$. Then the  support of $v$ is  an invariant
  set: $\supp(v) \in \mathcal{I}$.
\end{lem}

\begin{proof}
Let $f \in L^p$ be positive such that  $f\un_{\{v > 0\}} = v$,  and $g
\in L^q$ positive.  
We have: 
\[
k^{[g,f]}_T(\supp(v)^c, \supp(v)) = \langle g \un_{\{v = 0\}},
T(f\un_{\{v > 0\}}) \rangle = \langle g \un_{\{v = 0\}},  T(v) \rangle =
\rho(v)\,  \langle g \un_{\{v = 0\}}, v \rangle = 0,
\]
where we used that $f\un_{\{v>0\}} = v$ for the second equality and that
$v$ is an eigenfunction of $T$   with eigenvalue $\rho(v)$
for the third one.
This proves that the set $\supp(v)$ is $T$-invariant as the zeros of
the map $k^{[g,f]}_T$ does not depend on the choice of the positive
functions $f$ and $g$. 
\end{proof}

In some cases, invariance is the same for an operator and its resolvent. 
\begin{lem}[Resolvent of a positive operator]
  \label{lem:inversepositive}
  Let $T$ be a positive operator on $L^p$ with $p \in (1,
  +\infty)$. If $\lambda\in\mathbb{R}$ satisfies $\lambda > \rho(T)$,
  then the operator $\lambda \id - T$ is invertible, and its inverse
  is a positive operator on $L^p$. Moreover, the
  $(\lambda \id - T)^{-1}$-invariant sets are exactly the
  $T$-invariant sets.
\end{lem}
    \begin{proof}
      Since we have $\lambda> \rho(T)$, the operator
      $(\lambda \id - T)$ is invertible, and its inverse is given by
      its Neumann series:
      \[
        (\lambda \id - T)^{-1} = \sum\limits_{n = 0}^{+\infty} \lambda^{-n-1} T^n.
      \]
      This proves  both that  the operator $(\lambda  \id -  T)^{-1}$ is
      positive   and,  thanks   to  Lemma~\ref{lem:inv_T^n},   that  its
      invariant sets are exactly the $T$-invariant sets.
\end{proof}

Following  \cite{schwartz_61}, we consider the atoms associated to $T$. 

\begin{defi}[Admissible set and atoms]
  \label{def:atom}
Let $T$ be a positive operator on $L^p$ with $p \in  (1,+\infty)$.
A set which belongs to the $\sigma$-field $\ca=\sigma(\ci)$ generated by
the family $\ci$ of invariant sets is called
\emph{admissible}. A minimal admissible set with positive measure is
called an \emph{atom} of the operator $T$ or $T$-atom. 
\end{defi}

Notice that a set of positive measure $A$ is a $T$-atom if and only if it is an atom for the
measured space $(\Omega, \mathcal{A}, \mu)$. We denote by $\atom$  the set of
atoms:
\[
  \atom=\{A\in \ca\, \colon  \text{$A$ is a $
    T$-atom}\}.
\]
Since   atoms  have   positive  measure   and  the   measure  $\mu$   is
$\sigma$-finite, we  deduce that the  set $\atom$ is at  most countable.
When there  is no ambiguity on  the operator $T$, we  shall simply write
atom  for $T$-atom.   We  present  Example~\ref{ex:kernel2} below  where
there  is no  atom, and   Example~\ref{ex:adm_neq_bor}  where not  all
measurable sets are admissible.

\begin{ex}[The Volterra operator]\label{ex:kernel2}
  In  Example~\ref{ex:kernel}  on  the   Volterra  operator  $T_k$,  the
  admissible  $\sigma$-field is  the Borel  $\sigma$-field on  $[0, 1]$:
  $\ca=\cf$.   Notice that the  operator $T_k$ has no atom:
  $\atom=\emptyset$.
\end{ex}

  \begin{ex}[$\ca \neq \cf$]
  \label{ex:adm_neq_bor}
  We           consider           the           measured           space
  $(\Omega=[0,  1], \cf=\cb([0,  1]),  \leb)$, with  $\cf$ the  Borel
  subsets of $[0, 1]$ and $\leb$  the Lebesgue's measure on $[0, 1]$,
  and     the    kernel     $k$     on    $[0,     1]$    defined
  by:
  \begin{equation}
   \label{eq:def-k1}
  k (x,y) =  \un_{\{x \leq 1/2 \leq  y \leq x + 1/2\}}  + \un_{\{x \geq
    1/2\}} \un_{\{y \leq x - 1/2\}} \quad\text{(see Fig.~\ref{fig:kernel_1})}.
\end{equation}

Let $A \subset [0,1]$ be a measurable set. Then $A$ is $T_k$-invariant
if and only if for a.e.\ $x \in A^c \cap [0,1/2]$, we have
\(
\leb\left( \left[ 1/2, x+1/2 \right] \cap A \right) = 0
\)
and for a.e.\ $x \in A^c \cap [1/2, 1]$, we have
\(
\leb\left( \left[ 0, x-1/2 \right] \cap A \right) = 0
\).
Thus, $A$ is $T_k$-invariant if and only if for a.e.\
$x \in A^c \cap [0,1/2]$, we have
\(
\left[ 1/2, x+1/2 \right] \subset A^c \muae
\)
and for a.e.\ $x \in A^c \cap [1/2,1]$, we have
$\left[ 0, x-1/2 \right] \subset A^c \muae$.  Thus $A$ is
$T_k$-invariant if and only if $A = [a, 1/2] \cup [a + 1/2, 1]$ a.e.\
with $a \in [0,1/2]$.  Therefore the $\sigma$-field $\ca$ of
$T_k$-admissible sets consists in all the measurable sets which are
a.e.\ equal to $A \cup (A + 1/2)$ where $A \subset [0,1/2]$ is a
Borel set.  In particular, we have $\ca \neq \cf$.  Notice the
operator $T_k$ has no atom: $\atom=\emptyset$.
\end{ex}

\begin{figure}
\centering
\begin{subfigure}[b]{.4\textwidth}\centering
\includegraphics[width=1\linewidth]{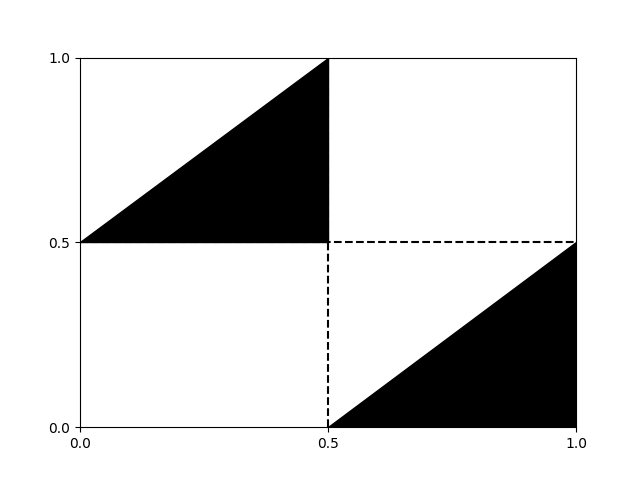}
\caption{Kernel $k$ defined in~\eqref{eq:def-k1}.}
\label{fig:kernel_1}
\end{subfigure}
\begin{subfigure}[b]{.4\textwidth}\centering
\includegraphics[width=1\linewidth]{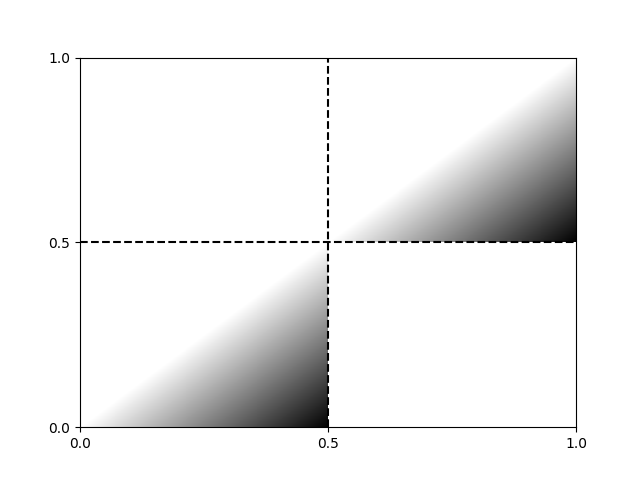}
\caption{Kernel $k^{\otimes 2}$ defined in~\eqref{eq:def-k12}.}
\label{fig:kernel_2}
\end{subfigure}
\caption{Example of some $[0, 1]$-valued kernels on $[0, 1]$.}
\label{fig:kernel}
\end{figure}

\subsection{Future and Past}
\label{sec:future-past}

We  now  consider   the  future  and  past  of  a   set,  and  refer  to
Remark~\ref{rem:epidem}      below      for      an      epidemiological
interpretation. Recall  the Definition~\ref{def:min-max} on  minimal and
maximal set.

\begin{defi}[Future and past]
  \label{defi:past-future}
    Let $T$  be a positive  operator on $L^p$ with $p \in  (1,+\infty)$. 
  Let $A$ be a measurable set.  We  define its  \emph{future},
$F(A)$, as the minimal invariant set  containing $A$ (that is, the minimal
set of $\cp=\{B\in \ci\, \colon\, A\subset B \muae\}$) and its
\emph{past}, $P(A)$, as the minimal co-invariant set  containing $A$. 
\end{defi}

We shall use later on the following  notation for  the future and past of a set $A$ without $A$:
\begin{equation}
  \label{eq:def-F*}
  F^*(A)=F(A)\cap A^c
  \quad\text{and}\quad
  P^*(A)=P(A)\cap A^c. 
\end{equation}

The next lemma ensures the existence of the future and the past.

\begin{lem}[Existence of future and past]
   \label{lem:past-future-existent}
Let $A\in \cf$, then its future  and its past   exist and
are unique, up to an a.e.\ equality.
\end{lem}

\begin{proof}
  We only consider the future, as the proof concerning the past is
  similar.  The set $\cp=\{B\in \ci\, \colon\, A\subset B \muae\}$ is
  stable by countable intersection thanks to
  Lemma~\ref{lem:union_int_inv}. Lemma~\ref{lem:min-exist} ensures the
  existence of a minimal set for $\cp$.  The uniqueness is also clear.
  This provide the existence of the future of $A$.
\end{proof}

Let us mention that the ``$k$-closure'' of a set for a kernel operator
$T_k$ introduced by Nelson \cite[p.~714]{nelson74} correspond to its
past (with respect to the invariant sets associated to $T_k$).  Let us
gather without proof a number of elementary facts.

\begin{lemme}[Basic properties of the future of a measurable set]
  \label{lem:elementary_past_future}

  For any measurable sets $A$ and $B$, and for any at most countable
  family of measurable sets $(A_i)_{i\in I}$, the following properties
  hold:
  \begin{enumerate}[(i)]
	\item $F(\emptyset) = \emptyset \muae$ and $F(\Omega) = \Omega \muae$.
	\item A set  $A$ is invariant if and only if  $F(A) = A \muae$.
	\item If $A \subset B \muae$, then
          $F(A) \subset F(B) \muae$. \label{item:fut_monotone}
	\item \label{item:union_fut}
          $F \left( \bigcup_{i \in I} A_i \right) = \bigcup_{i \in I}
          F(A_i) \muae$.
        \item \label{item:inter_fut}
          $F \left( \bigcap_{i \in I} A_i \right) \subset
          \bigcap_{i \in I} F(A_i) \muae$; 
          the reverse inclusion does not hold in general.
    \item \label{item:FFA} $F(F(A)) = F(A) \muae$.
    \end{enumerate}
    The properties (i-vi) also hold with $F$ replaced by $P$. 
  \end{lemme}

  Futures and pasts are related by the following elementary result; by
  contrast, note that the inclusion $A\subset F(B)$ does not imply
  that $B\subset P(A)$ in general, see  Example~\ref{ex:finite}.
 
  \begin{lem}[Intersection of a future and a past]
    \label{lem:darknessoffuturepast}
  Let $A, B$ be two measurable sets. We have:
  \[
    A \cap P(B) = \emptyset \muae \iff F(A)\cap P(B) = \emptyset \muae
    \iff F(A)\cap B = \emptyset \muae.
  \]
\end{lem}

\begin{proof}
  If $A\cap P(B) = \emptyset$, then $A$ in included in $
  P(B)^c$. Since the set $P(B)^c$ is invariant, we have
  $F(A)\subset P(B)^c$ by minimality, which means that
  $F(A)\cap P(B) = \emptyset$. The converse is clear since
  $A\subset F(A)$. The second equivalence is proved similarly.
\end{proof}

\subsection{Irreducibility}
\label{sec:irreduc}

Similarly to Schaefer  \cite[Definition~8.1]{schaefer_74}, we can define
the irreducibility of an operator in terms of invariance.

\begin{defi}[Irreducible operators and invariant sets]\label{defi:irre}
  Let $T$ be a positive operator on $L^p$ with $p \in  (1,+\infty)$.
\begin{enumerate}[(i)]
   \item    The operator  $T$ is \emph{irreducible} if its only
     invariant sets are$\muae$ equal to $\emptyset$ or $\Omega$.
   \item   The  measurable   set  $A$   is  $T$-irreducible   or  simply
     \emph{irreducible} if it is measurable with positive measure and if
     the restricted operator $T|_A$ on $L^p(A)$ is irreducible.
\end{enumerate}
\end{defi}

We refer to  Lemma~\ref{lem:irr+at} and Theorem~\ref{th:equi_atomes} for
relations    between   irreducible    sets   and    atoms.   See    also
Example~\ref{ex:irr-T2} for a comment on the irreducible sets of $T$ and
of $T^2$.  We  now state explicitly the relation  between invariance and
irreducibility      from      Section~\ref{sec:banach}     and      from
Definitions~\ref{defi:inv}  and~\ref{defi:irre}.  Recall  the definition
of the closed ideal  in~\eqref{eq:def-LpA}.

\begin{lem}\label{lem:equi_def_inv/irre}
  Let $T$ be a positive operator on $L^p$ with $p \in (1,+\infty)$. Then the operator $T$ is
  irreducible if and only if it is ideal-irreducible.
\end{lem}

\begin{proof}
  It is a direct consequence of Lemma~\ref{lem:inv_sets_ideals} and
  the fact that the closed ideals of $L^p$ are exactly given by
  $L^p_A$ for $A$ measurable, see \cite[Section~2]{vovo2016} and
  \cite[Theorem 5.14, p.~94]{schaefer_74}.
\end{proof}

\subsection{The countable case and an underlying preorder}
\label{subsec:countable}

We assume in  this section only that $\Omega$ is  at most countable, and
without loss of generality that  $ \mu(\{x\})  > 0$ for all $x\in \Omega$.
Let  $T$ be a positive
operator on $L^p$.  The the map  $k_T$ is entirely defined by the values
of $k_T(\{x\}, \{y\})$,  denoted $k_T(x,y)$, for $x,  y\in \Omega$.  The
notions of  admissibility, atoms,  invariance and irreducibility  may in
that  case be  completely  understood by  studying  a particular  binary
relation on  $\Omega$ given  in terms  of $k_T$. To  see this,  we write
$x\preccurlyeq  y$ if  $ x  =  y$ or  if  there exists  $n\in \N^*$  and
$   (x=x_0,x_1,...,x_{n-1},x_n  =   y)  \in   \Omega^{n+1}$  such   that
$ \prod_{i  = 1}^n k_T(x_{i-1},  x_i) >0$.  The  relation $\preccurlyeq$
defines a preorder  on $\Omega$ (that is, a  reflexive transitive binary
relation).                          The                         relation
$x\sim y \iff  (x\preccurlyeq y \text{ and } y  \preccurlyeq x)$ is then
an equivalence  relation.  The equivalence classes  of $\sim$ correspond
to atoms of the operator  $T$, and the preorder $\preccurlyeq$ naturally
induces  a (partial)  order on  them: for  two atoms  $A$, $B$,  we have
$A \preccurlyeq B $ if $x\preccurlyeq y $ for all $x\in A$ and $y\in B$.
The admissible sets  are the sets $A$  that may be written  as unions of
atoms   (the  $\sigma$-field   $\ca$  is   generated  by   the  set   of
atoms).  Furthermore, a  set $A$  is invariant  if and  only if  the two
following conditions hold:
\begin{itemize}[-]
    \item $A$ is the union of atoms $(A_i)_{i\in I} $ (in particular, $A$ is admissible), 
    \item The family $(A_i)_{i\in I} $  is a downset  for the order
      induced by $\preccurlyeq$ on atoms.
\end{itemize}

For a set $A$, its future corresponds to the downward closure of $A$,
that is, the smallest downset containing $A$,  and its future and past
are given by:
\[
  F(A)=\bigcup_{x\in A} \{y\in \Omega\, \colon\, y\preccurlyeq  x\}
  \quad\text{and}\quad
  P(A)=\bigcup_{x\in A} \{y\in \Omega\, \colon\, x\preccurlyeq  y\}.
\]
  Notice the definition of atoms, invariant sets, future and past of a
  set depends only on the support $\{k_T>0\}\subset \Omega^2$ of the
  kernel $k_T$.

\begin{rem}[Epidemiological interpretation]
   \label{rem:epidem}
   In the epidemiological interpretation where each element of $\Omega$
   is seen as an individual or an homogeneous sub-population and $T$ can
   be assimilated to the next generation operator, we  have:
   \begin{itemize}[-]
   \item $k_T(x,y)>0$ means that individual  $y$ can directly infect individual $x$;
   \item $x\preccurlyeq y$ when there may  be a chain of infections from
     individual $y$ to individual $x$;
   \item the  set $A$ is invariant  if an epidemic started  in $A$ stays
     within $A$;
   \item the future $F(A)$  of $A$ is the set of all  individuals that might
     get infected by  an epidemic starting at every individual  of $A$;
   \item  the past $P(A)$ of $A$ is the set of all individuals  that might
     infect an individual of $A$.
   \end{itemize} 
\end{rem}

In Section~\ref{sec:cvxe} we shall consider convex sets, that is, sets
$A$ such that $A = F(A) \cap P(A)$. They have a simple representation
when $\Omega$ is at most countable.  Following the terminology
of \cite[Section I.4, p.~7]{Bir67}, for $x, y\in \Omega$, we define 
the \emph{interval}
$[x,y]=\{z\in \Omega\,\colon\, x\preccurlyeq z \preccurlyeq y \}$,
and say that a set $A\subset \Omega$ is \emph{(order-)convex} if:
\[
x,y\in A \quad \implies \quad [x,y]\subset A.
\] 
It is easily checked that an order-convex set corresponds to being the
union of atoms $(A_i)_{i \in I}$ where the family $(A_i)_{i \in I}$ is
order convex, that is, if $A$ is an atom such that $A_i
\preccurlyeq A \preccurlyeq A_{i'}$ for some $i, i'\in I$, then $A$
belongs to the family  $(A_i)_{i \in I}$. 

\begin{figure}
\centering
\begin{subfigure}[b]{.4\textwidth}\centering
$\begin{pmatrix}
0 & \star & 0 & 0 & 0 & 0 \\
\star & 0 & \star & 0 & 0 & 0 \\
0 & \star & 0 & 0 & 0 & 0 \\
0 & \star & 0 & 0 & 0 & 0 \\
0 & \star & 0 & 0 & 0 & 0 \\
0 & 0 & 0 & \star & \star & 0 \\
\end{pmatrix}$
\caption{Matrix on $\{1,
 \ldots, 6\}$ with $\star > 0$.}
\label{fig:m-graph}
\end{subfigure}
\begin{subfigure}[b]{.4\textwidth}\centering
  \begin{tikzpicture}[v/.style={circle,draw}]
    \node (2) [v] {2};
    \node (1) [v,left=of 2] {1};
    \node (3) [v,right=of 2] {3};
    \node (4) [v,below left=of 2] {4};
    \node (5) [v,below right=of 2] {5};
    \node (6) [v,below right=of 4] {6};

    \path[>=latex,->]  (2)  edge (4)      edge (5)
                            edge[<->] (1) edge[<->] (3)
                       (5) edge (6)
                       (4) edge (6);
                     \end{tikzpicture}
                     
\caption{Associated communication graph.}
\label{fig:g-graph}
\end{subfigure}

\caption{Matrix and associated communication graph 
  from Example~\ref{ex:finite}.}
\label{fig:graph}
\end{figure}
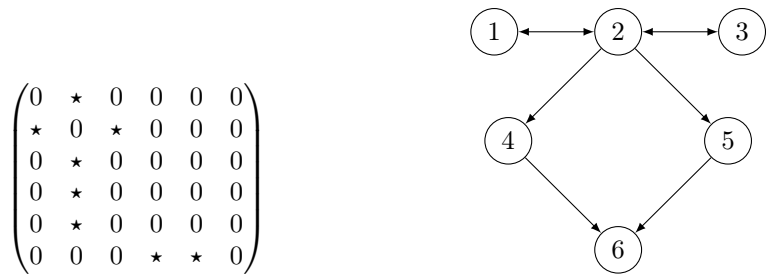

\begin{ex}[A finite elementary case]
  \label{ex:finite}
We consider the finite case:  $\Omega=\{1, \ldots, n\}$ with $n\in \N^*$,   $\mu$ is the
counting measure,  $L^p(\Omega)$ is identified with $\R^n$
and operators  on $L^p$ with  $n \times n$ real matrices.  A
 matrix $M = (M_{i,j})_{1 \leq i,j \leq n}$ with
nonnegative entries is alternatively represented by the oriented weighted graph
$G = (V, E)$ with $V = \{1, \dots, n\}$ and with a weight $M_{i,j}$ to
the edge $(j,i) \in E$.

  To illustrate, consider the case $n=6$ with the matrix given in
  Fig.~\ref{fig:m-graph} where the $\star$ correspond to positive terms.
  The corresponding communication graph (an oriented edge is represented
  for   each   positive    entry   of   the   matrix)    is   given   in
  Fig.~\ref{fig:g-graph}.   The  atoms  are:  $\{1,  2,  3\}$,  $\{4\}$,
  $\{5\}$ and $\{6\}$.  The invariant sets are: $\Omega$, $\{4, 5, 6\}$,
  $\{4, 6\}$, $\{5, 6\}$, $\{6\}$ and $\emptyset$.  For example the sets
  $\{1,2,3\}$, $\{1, 2\}$  and $\{1\}$ are irreducible,  and among those
  three  only  the  first  one  is admissible.   For  example  the  sets
  $\{1,2,3,4\}$, $\{5\}$ and $\{5,6\}$ are  convex.  Even though the set
  $\{5\}^c$ is admissible, it is not convex.

Let us notice that the inclusion  in 
  Lemma~\ref{lem:elementary_past_future}~\ref{item:inter_fut} is not
  an equality in general; indeed we have: $F(\{4\} \cap \{5\}) = F(\emptyset) = \emptyset$ whereas
  $F(\{4\}) \cap F(\{5\}) = \{6\}$.
Notice also that $\{5\}$ belongs to the future of $\{1, 2, 3,4\}$, but
the latter does not belong to the past  of $\{5\}$. 
\end{ex}

The  countable state  space  $\Omega$ and  the  above representation  of
convex sets will  guide many definitions and proofs  below.  The general
case is at  the same time more technical (invariant  sets are defined up
to an  a.e.\ equality), and more  subtle: for example, the  union of all
atoms may be a  strict subset of the whole space; it  may even be empty,
as  in  Example~\ref{ex:kernel2}  where  there exists  no  atom  of  the
Volterra operator.  For  this reason we will work only  on invariant and
co-invariant sets, viewing  them intuitively as down- and  up-sets of an
underlying order that we will not write down formally.

\subsection{Order-convex subsets}
\label{sec:cvxe}

By construction of the future and the past, a measurable set $A$ is always
included in $F(A)\cap P(A)$. The set $A$ is convex when there is equality. 

\begin{defi}[Order-convex subset]\label{def:convex}
  Let $T$ be  a positive operator on $L^p$ with  $p \in (1,+\infty)$.  A
  set $A$ is \emph{order-convex} for $T$, or $T$-\emph{convex}, if it is
  measurable and $A = F(A) \cap P(A) \muae$.
\end{defi}

When there is no ambiguity on the operator $T$, we shall simply write
convex for $T$-convex. 

\begin{ex}[Convex sets of the Volterra operator]\label{ex:cvx_volterra}
  We continue  Example~\ref{ex:kernel} on  the Volterra  operator. Using
  the description  therein of  invariant and  co-invariant sets,  we get
  that  a  set  $A$ is  convex  if  and  only  if $A=[a,b]$  a.e.\  with
  $0\leq a\leq b\leq 1$.
\end{ex}

\begin{rem}[Atoms, irreducibility and convexity coincide for $T$ and $T^\star$]
  \label{rem:T=T*}
  Notice that the admissible $\sigma$-field is the same for the operator
  $T$ and its dual $T^\star$.   Thanks to~\eqref{eq:k-adj}, the operator $T$
  is irreducible if and only if $T^\star$  is irreducible. Thus a set $A$ is
  a $T$-atom (resp. $T$-irreducible, $T$-convex) if  and only if it is a
  $T^\star$-atom (resp. $T^\star$-irreducible, $T^\star$-convex).
\end{rem}

\begin{rqe}[Convex sets on a countable measurable set]
  We go back to the framework of subsection~\ref{subsec:countable},
  where $\Omega$ is an at most countable set. Then $A$ is a convex set
  in the sense of Definition~\ref{def:convex} if and only if $A$ is
  order-convex in the sense of the definition of
  subsection~\ref{subsec:countable}. Therefore the two definitions are
  coherent.
\end{rqe}

Recall~\eqref{eq:def-F*}, where we set $F^*(A)=F(A) \cap A^c$ and
$P^*(A)=P(A) \cap A^c$.

\begin{lem}[Characterization of convexity]
  \label{lem:equi_convex}
Let $A$ be a measurable set. The following properties are equivalent:
\begin{enumerate}[(i)]
    \item $A$ is convex. \label{lem:item:interval}
    \item $F^*(A)\cap P^*(A)  = \emptyset \muae$. \label{lem:item:empty_inter}
    \item $F^*(A)$ is invariant. \label{lem:item:AF_inv}
    \item $P^*(A)$ is co-invariant. \label{lem:item:AP_co}
    \item There exist an invariant set $B$ and a co-invariant set $C$ such that $A = B \cap C \muae$. \label{lem:item:inter_ic} 
\end{enumerate}
\end{lem}

As a particular consequence of~\ref{lem:item:inter_ic}, we get that if $A$
is measurable then $F(A)  \cap P(A)$ is convex. 
We illustrate in Fig.~\ref{fig:contam} the possible connections between
the sets $A$, $F^*(A)$, $P^*(A)$ and the complementary of their union,
when $A$ is convex. 

\begin{proof}
  Use the definition of convexity and that
  Point~\ref{lem:item:empty_inter} is equivalent to
  $P(A)\cap F(A) \cap A^c = \emptyset$ to get that
  Points~\ref{lem:item:interval} and~\ref{lem:item:empty_inter} are
  equivalent.  Clearly Point~\ref{lem:item:interval} implies
  Point~\ref{lem:item:inter_ic}.  The proofs
  involving~\ref{lem:item:AF_inv} are similar to the ones
  involving~\ref{lem:item:AP_co}, so the latter are left to the
  reader.

\medskip

We assume Point~\ref{lem:item:empty_inter} and prove Point
\ref{lem:item:AF_inv}. As $F^*(A) \cap P^*(A) = \emptyset$, the set
$F^*(A)$ is a subset of $P^*(A)^c$.  Therefore, the set
$F^*(A) = (A \cup F^*(A)) \cap (A \cup P^*(A))^c = F(A) \cap P(A)^c$
is invariant as the intersection of two invariant sets. Thus
Point~\ref{lem:item:AF_inv} holds.

Conversely, assuming Point~\ref{lem:item:AF_inv}, the set $F^*(A)$ is
invariant, so the set $P(A) \cap F^*(A)^c $ is a co-invariant set
containing $A$ and included in $P(A)$. By minimality of $P(A)$, this
set is equal to $P(A)$, thus $P(A)\subset F^*(A)^c$.  This gives
Point~\ref{lem:item:empty_inter}.

\medskip

Finally let us assume Point~\ref{lem:item:inter_ic} and prove Point
\ref{lem:item:interval}.  By assumption, we have $A = B \cap C$ with
$B$ invariant and $C$ co-invariant.  By minimality, we get that
$F(A)\subset B$ and $P(A)\subset C$, and thus:
\[
  A\subset\, F(A) \cap P(A) \, \subset B \cap C =A.
\]
This gives that $A$ is convex, that is, Point~\ref{lem:item:interval}.
\end{proof}

\begin{figure}
\centering
\begin{tikzpicture}[v/.style={circle,draw},node distance=3mm]
    \node (pstar) [v,minimum size=1.4cm]                {$P^*(A)$};
    \node (a)     [v,below=of pstar,minimum size=1.4cm] {$A$};
    \node (fstar) [v,below=of a,minimum size=1.4cm]     {$F^*(A)$};
    \node (a0)    [v,right=of a,minimum size=1.4cm]     {$A_0$};
    \path[>=latex,->] (pstar) edge[loop right] (pstar)
                              edge[out=-180, in=180] (fstar)
                              edge (a) edge (a0)
                      (a) edge[loop left] (a)
                          edge (fstar)
                      (a0) edge[loop right] (a0)
                           edge (fstar)    
                           (fstar) edge[loop right] (fstar);
     \tikzset{baseline={(pstar)}}
     \end{tikzpicture}%
 \begin{minipage}[t]{0.4\linewidth}
   \small The three sets $A$, $F^*(A)$  and $P^*(A)$ are disjoint as $A$
   is convex.  Let $A_0=(P(A) \cup F(A))^c$,  so that the four sets $A$,
   $F^*(A)$,  $P^*(A)$  and  $A_0$  form  a  partition  of  $\Omega$  in
   admissible sets.  The possible connections  between the four sets are
   depicted in the  picture: if there is  no arrow from $B$  to $C$ then
   $k_T(C,B) = 0$.
  \end{minipage}
\caption{Past and future for a $T$-convex set $A$.}
\label{fig:contam}
\end{figure}

\medskip

We end this section with an auxiliary result on convexity.

\begin{lem}[Intersection of convex and invariant sets]
  \label{lem:conv-inv}
Let $A$ be  a convex set and $B$ an invariant set. Then the set $A \cap
B$ is convex. 
\end{lem}

\begin{proof}
  We have $A\cap B = P(A) \cap  F(A) \cap B$ by definition of
  convexity. So  by  Lemma~\ref{lem:equi_convex}  it  is  convex as  the
  intersection  of the  co-invariant set  $P(A)$ with  the invariant  set
  $F(A)\cap B$.
\end{proof}

\subsection{Properties of the restricted operators}

Let $\Omega'\subset \Omega$ be a measurable set with positive measure.
Let $T$ be a positive operator on $L^p$ with $p\in (1, +\infty )$.  We
start with a result of stability of invariant/irreducible sets and
atoms by restriction.
Recall $T_{\Omega'}$  is the restriction of $T$
to $\Omega'$ given by~\eqref{eq:def-TA}.

\begin{lem}[Restriction and invariance/irreducibility]
\label{lem:inv_stab_rest}
Let  $T$ be  a positive  operator on  $L^p$ with  $p\in (1,  +\infty )$,
$\Omega'\subset  \Omega$ a  measurable  set with  positive measure,  and
$T'=T_{\Omega'}$  the  restriction of  $T$  on  $\Omega'$. We  have  the
following properties.
\begin{enumerate}[(i)]
\item \label{item:inv_stab_O} The set $\Omega'$ is $T'$-invariant and
  $T'$-co-invariant. 
\item \label{item:inv_stab_inv} Every $T$-invariant set  is $T'$-invariant.
 \item \label{item:inv_stab_adm} One can replace invariant
   in~\ref{item:inv_stab_inv} by co-invariant and by admissible. 
  \item  \label{item:inv_stab_irr}  The set   $A\subset \Omega'$  is
  $T$-irreducible if and only if it is $T'$-irreducible.

\item  \label{item:inv_stab_inv2} If   $\Omega'$  is  $T$-invariant and
  $A \subset \Omega'$, then    $A$  is
  $T$-invariant if and only if it is $T'$-invariant.
\end{enumerate}
\end{lem}

\begin{proof}
  Since $k_{T'}(\Omega'^c, \cdot)=k_{T'}(\cdot, \Omega'^c)=0$, we obtain
  Point~\ref{item:inv_stab_O}. Recall the definition of $\cz_T$, the set
  of zeros of $k_T$, given in~\eqref{eq:def-ZT}.  Since $T$ is positive,
  we clearly  have $k_T\geq k_{T'}$ and  thus $\cz_{T}\subset \cz_{T'}$.
  This  gives Point~\ref{item:inv_stab_inv}  and  the co-invariant  case
  in~Point~\ref{item:inv_stab_adm}.  As the  invariant sets
  generates  the  $\sigma$-field  of  the admissible  sets, we  get  the
  admissible           case           of~Point~\ref{item:inv_stab_adm}.
  Point~\ref{item:inv_stab_irr}    is    immediate.    We    have    for
  $A\subset \Omega'$:
\[
  k_T(A^c, A)=k_T(A^c \cap \Omega', A)+ k_T(A^c \cap \Omega'^c, A)
  \leq  k_{T'} (A^c, A)+ k_T(\Omega'^c, \Omega'). 
\]
If  $\Omega'$ is  invariant,  and thus  $k_T(\Omega'^c, \Omega')=0$,  we
deduce that if $A$ is $T'$-invariant, then it is $T$-invariant. This and
Point~\ref{item:inv_stab_inv} give Point~\ref{item:inv_stab_inv2}.
\end{proof}

We now study more the stability  of convexity and future by restriction.
Let $F'(A)$ denote the future of the measurable set $A$ for the operator
$T'=T_{\Omega'}$.

\begin{lem}[Restriction and convexity/future]
  \label{lem:rest_prop}
  Let $T$  be a positive  operator on $L^p$  with $p\in (1,  +\infty )$,
  $\Omega'\subset \Omega$ be a measurable set with positive measure, and
  $T'=T_{\Omega'}$  be  the  restriction   of  $T$  on  $\Omega'$.
  For any measurable set
  $A\subset  \Omega'$, the following
  properties hold.
  \begin{enumerate}[(i)]
  \item\label{lem:interval_restriction}
    If $A$ is $T$-convex  then it is  $T'$-convex.
  \item \label{lem:future_as_a_union}
    We have a.e.:
    \begin{equation}
      \label{eq:F=F'}
      F(A) = F(  F(A) \cap \Omega'^c )\, \cup F'(A) . 
    \end{equation}
  \item\label{lem:future_as_a_union_2}
    If $\Omega'$ is $T$-convex, then we have
    $F'(A) = F(A) \cap \Omega'$ a.e.. In particular, $T'$-invariant
    subsets of $\Omega'$ are exactly the trace on
   $\Omega'$ of $T$-invariant sets. 
\end{enumerate}
\end{lem}

\begin{proof}
  Let $A \subset \Omega'$ be measurable sets.  As $F(A)$ is
  $T$-invariant, then by
  Lemma~\ref{lem:inv_stab_rest}-\ref{item:inv_stab_inv}, we get that
  the set $F(A) \cap \Omega'$ is $T'$-invariant, and similarly the set
  $P(A) \cap \Omega'$ is $T'$-co-invariant. Since they both contain
  $A$, we deduce by the definition of the future and past of a set,
  that:
\begin{equation}
  \label{eq:F'A-FA}
  F'(A) \subset F(A) \cap \Omega'
  \quad\text{and}\quad
  P'(A) \subset P(A) \cap \Omega'.
\end{equation}
If       $A$       is       $T$-convex,       we       deduce       that
$A \subset  P'(A) \cap F'(A)  \subset P(A)\cap  F(A) = A$.  This implies
that $A$ is $T'$-convex, that is Point~\ref{lem:interval_restriction}.
    
\medskip

We prove Point~\ref{lem:future_as_a_union}.  Setting
$B = F(A) \cap \Omega'^c $ and $C= F(B)\cup F'(A) $, the goal is
to prove that $C = F(A)$. We shall first
prove that $C$ is $T$-invariant. Thanks to~\eqref{eq:F'A-FA}, we have
$ F(A) \cap \left( \Omega' \cap F'(A)^c\right)^c =\left( F(A) \cap
  \Omega'^c \right) \cup F'(A) \subset C$, that is:
\begin{equation}
   \label{eq:mino-C}
   C^c \subset F(A)^c \cup \left(\Omega' \cap F'(A)^c\right). 
\end{equation}
We deduce that:
\begin{align*}
  k_T(C^c, C)
  &\leq  k_T(C^c, F(B))+ k_T(F(A)^c, F'(A))+ k_T(\Omega' \cap F'(A)^c
    , F'(A))\\
  &\leq  k_T(F(B)^c, F(B))+ k_T(F(A)^c, F(A))+ k_{T'}(F'(A)^c
    , F'(A))\\
  &=0,        
\end{align*}
where we used the additivity and monotonicity of $k_T$
and~\eqref{eq:mino-C} for the first inequality; the monotonicity of
$k_T$, $F(B)\subset C$,~\eqref{eq:F'A-FA} (twice) and the definition
of $T'$ for the second; that $F(B)$ and $F(A)$ are $T$-invariant, and
$F'(A)$ is $T'$-invariant for the last equality.  Thus, the set $C$ is
$T$-invariant.  As $A\subset C \subset F(A)$ (use
$A \subset F'(A) \subset C$ for the first inclusion, and
$C \subset F(F(A)) \cup F(A)=F(A)$ for the second, see
Lemma~\ref{lem:elementary_past_future}~\ref{item:FFA}
and~\eqref{eq:F'A-FA}), we deduce by minimality of the future that
$C=F(A)$. This gives Point~\ref{lem:future_as_a_union}.

\medskip

We  now prove  Point~\ref{lem:future_as_a_union_2}.  Since $\Omega'$  is
$T$-convex, we have:
\[
  F(A) \cap \Omega'^c = F(A) \cap (F(\Omega') \cap P(\Omega'))^c
  = F(A) \cap (F(\Omega')^c \cup P(\Omega')^c).
\]
Since $F(A) \subset F(\Omega')$, we deduce that:
\[
  F(A) \cap \Omega'^c  = F(A) \cap P(\Omega')^c,
\]
which is invariant
as         intersection      of two     invariant      sets.       Now,
using~\ref{lem:future_as_a_union},     we get that
$F(A) =  (F(A)\cap \Omega'^c)\, \cup F'(A) $.  Taking the intersection with
$\Omega'$ yields that $F(A)\cap \Omega'=F'(A)$.
This ends the proof. 
\end{proof}

\subsection{Properties of atoms}
\label{sec:prop-atom}
We first prove that atoms are convex and irreducible. 

\begin{lemme}\label{lem:atom_convex}
Atoms are  convex.
\end{lemme}
\begin{proof}
  Let $A$ be an atom and set $B = F(A) \cap P(A)$. We consider the
  family of measurable sets
  $\ca'= \{C \in \mathcal{F}\, \colon\, C \cap A = \emptyset \muae
  \text{ or } B \subset C\muae\}$.  For simplicity we do not write
  a.e.\ anymore in this proof.  Let $C$ be an invariant set.  As $A$
  is a minimal admissible set, we have $C \cap A = \emptyset$ or
  $A \subset C$. In the latter case, by minimality of $F(A)$, as $C$
  is invariant, we deduce that $F(A)\subset C$, and thus $B\subset C$.
  In any case, we get that $C$ belongs to $\ca'$, and thus $\ca'$
  contains all the invariant sets, that is $\ci\subset \ca'$.  A
  similar argument implies that $\ca'$ contains all the co-invariant
  sets, that is the complementary of all the invariant sets.

It  is clear  that $\ca'$  is stable  by countable  union and  countable
intersection.   Therefore, by  \cite[Theorem~4.2,  p.~130]{aliprantis},
$\ca'$  contains   the  $\sigma$-field  generated  by   $\ci$,  that  is
$\ca\subset \ca'$.  In  particular, the set $A$ belongs to $\ca'$.   As $A$ is an
atom  it has  positive measure.  This gives  that $B\subset  A$. As
$A\subset F(A) \cap P(A)$, we deduce that $B=A$, that is, the set $A$ is
convex.
\end{proof}

\begin{lem}
  \label{lem:atom_irr}
  Atoms are irreducible. 
\end{lem}
\begin{proof}
  Let $A$ be an atom. It is convex according to
  Lemma~\ref{lem:atom_convex}. Set $T'=T_A$. Let $B\subset A$ be
  $T'$-invariant (and thus $T|_A$-invariant), and
  denote its future with respect to $T'$ by $F'(B)$. By
  Lemma~\ref{lem:rest_prop}~\ref{lem:future_as_a_union_2}, we deduce
  that $B=F'(B)=F(B) \cap A$. This implies that $B$ is $T$-admissible. Since
  $A$ is an atom, we get that $B=A$ or $B=\emptyset$. This implies that
  $T|_A$ on $L^p(A)$ is irreducible, that is, $A$ is irreducible.   
\end{proof}

We then prove that intersections of irreducible sets with
admissible sets are trivial. 
\begin{lem}[Intersection of irreducible and admissible  sets]
  \label{lem:01law}
  If $A$ is admissible  and $B$ irreducible, then either
  $A\cap B = \emptyset \muae$ or $B \subset A \muae$.
\end{lem}

\begin{proof}
  Let $B$ be irreducible.  Assume first the set $A$ is invariant.
  According to
  Lemma~\ref{lem:inv_stab_rest}~\ref{item:inv_stab_O}-\ref{item:inv_stab_inv}
  with $\Omega'=B$ and Lemma~\ref{lem:union_int_inv}, the intersection
  $A\cap B$ is invariant for the operator $T_B$, and thus also for the
  restricted operator $T|_B$ on $L^p(B)$.  Since $B$ is irreducible,
  we deduce that $A \cap B = \emptyset \muae$ or $A \cap B = B
  \muae$. Thus the collection of sets whose intersection with $B$ is trivial,
  that is,   $\ca'= \{C \in \mathcal{F}\, \colon\, C \cap B = \emptyset \muae
  \text{ or } B \subset C\muae\}$, contains all invariant sets.

  It  is clear  that $\ca'$  is stable  by countable  union
  and  complement,  so  it  contains the  $\sigma$-field  $\ca$  of  the
  admissible  sets which  is generated  by the  invariant sets,  that is
  $\ca\subset \ca'$. Thus  the set $A $ belongs to  $\ca'$ and satisfies
  $A \cap B = \emptyset$ or $B \subset A$.
\end{proof}

We directly deduce from the previous lemma the following result. 

\begin{lem}[Irreducibility and atoms I]
  \label{lem:irr+at}
  All irreducible admissible sets are atoms.  
\end{lem}

We then prove that any irreducible set is a subset of an atom.

\begin{lem}[Irreducibility and atoms II]
  \label{lem:irr_in_atom}
  If $A$ is irreducible, then $F(A) \cap P(A)$ is an atom (which
  contains $A$ a.e.). 
\end{lem}

\begin{proof}
  Let $A$  be irreducible (and  thus measurable with  positive measure).
  Set $A' = P(A) \cap F(A)$.  Let $B \subset A'$ be $T$-invariant.  Then
  by   Lemma~\ref{lem:inv_stab_rest}~\ref{item:inv_stab_O}-\ref{item:inv_stab_inv},
  we   obtain  that $A   \cap  B$   is
  $T_A$-invariant,  so   by  irreducibility   of  $A$  we   have  either
  $A \subset B$  or $A \cap B  = \emptyset$.  If $A \subset  B$, then we
  have  $F(A)  \subset  F(B)=B\subset  A'  \subset F(A)$  as  $B$  is  a
  $T$-invariant  set  contained in  $A'$,  so  we  have  $B =  A'$.   If
  $A  \cap  B   =  \emptyset$,  then  the  set  $P(A)   \cap  B^c  $  is
  $T$-co-invariant and contains  $A$, so we have $P(A) \cap  B^c = P(A)$
  which implies  that $B  = \emptyset$ as  $B\subset A'\subset  P(A)$ by
  hypothesis.  This proves that $A'$  is irreducible.  Since $A'$ is
  admissible, we deduce from Lemma~\ref{lem:irr+at} that $A'$ is an
  atom. 
\end{proof}

To end this section we complete the statement of
Lemma~\ref{lem:inv_stab_rest} by considering atoms.  
Recall $T_{\Omega'}$  is the restriction of $T$
to $\Omega'$ given by~\eqref{eq:def-TA}.

\begin{prop}[Restriction and atoms]
  \label{prop:T-atom}
  Let $T$ be a positive operator on $L^p$ with $p\in (1, +\infty )$,
  $\Omega'\subset \Omega$  a measurable set with positive measure,
  and $T'=T_{\Omega'} = $  the
  restriction of $T$ on $\Omega'$. Let $A \subset \Omega'$ be measurable.
  \begin{enumerate}[(i)]
  \item \label{item:Ta-T'a}
    If $A$ is a $T$-atom then it is
     a $T'$-atom.
   \item \label{item:T'a-Ta}
     Assume $\Omega'$ is admissible. Then   $A$ is a $T'$-atom if and
     only if  it is
     a $T$-atom.
  \end{enumerate}
\end{prop}

\begin{rqe}[Open question]
  We conjecture the following result, which would imply~\ref{item:T'a-Ta}:
  if $\Omega'$ is admissible, then $A\subset\Omega'$
  is $T'$-admissible if and only if it is $T$-admissible. 
 \end{rqe}

\begin{proof}
  We  first prove  Point~\ref{item:Ta-T'a} Let  $A\subset \Omega'$  be a
  $T$-atom. It  has a  positive measure, and  it is  $T$-irreducible and
  $T$-convex by Lemmas~\ref{lem:atom_convex} and~\ref{lem:atom_irr}.  It
  is then $T'$-irreducible and $T'$-convex (and thus $T'$-admissible) by
  Lemmas~\ref{lem:inv_stab_rest}~\ref{item:inv_stab_irr}
  and~\ref{lem:rest_prop}~\ref{lem:interval_restriction}. Thus,  it is a
  $T'$-atom by Lemma~\ref{lem:irr+at}.

  \medskip

  We now prove Point~\ref{item:T'a-Ta}. Let $A$ be a $T'$-atom. It has a
  positive measure,  and it is $T'$-irreducible. It is also
  $T$-irreducible  by
  Lemma~\ref{lem:inv_stab_rest}~\ref{item:inv_stab_irr}. This implies
  that $F(A)\cap P(A)$ is a $T$-atom by
  Lemma~\ref{lem:irr_in_atom}. Since $\Omega'$ is admissible and
  $A\subset \Omega'$, we deduce that $F(A)\cap P(A)\subset
  \Omega'$. Thus $F(A)\cap P(A)$ is a $T'$-atom by
  Point~\ref{item:Ta-T'a}. It contains $A$, thus it is equal to $A$. This
  proves that $A$ is a $T$-atom.
\end{proof}

\subsection{A characterization of atoms}
\label{chara-atom}
The main goal of this subsection is to prove the following theorem, that links the definitions of atoms, convex  and irreducible sets.

\begin{theo}[Equivalent definitions of atoms]\label{th:equi_atomes}
  Let $T$ be a positive operator on $L^p$ with $p\in (1, +\infty )$. The  following
  properties are equivalent.

\begin{enumerate}[(i)]
\item \label{th:item:atom}
  The set $A$ is an atom. 
\item \label{th:item:min_interval}
  The  set $A$ is a  minimal convex set
  with positive measure.
 \item  \label{th:item:max_irre2}
    The    set    $A$     is    an admissible   irreducible  set. 
\item \label{th:item:max_irre}
   The set $A$ is a maximal irreducible          set. 
\end{enumerate}
\end{theo}

We first gives another link between convexity and irreducibility before
proving the theorem. 

\begin{lem}[Convexity and irreducibility]
  \label{lem:irr+cvx}
   A  minimal  convex set with positive measure is irreducible.
\end{lem}

\begin{proof}
  Assume that $A$ is minimal convex.  Let $B \subset A$ be a
  $T_A$-invariant set.  By
  Lemma~\ref{lem:rest_prop}~\ref{lem:future_as_a_union_2} (with
  $\Omega'=A$), we have $B = F(B)\cap A$, and thus $B$ is convex by
  Lemma~\ref{lem:conv-inv}.  Therefore we have $B=A$ or $B=\emptyset$
  by minimality.  This proves that the set $A$ is irreducible.
\end{proof}

\begin{proof}[Proof of Theorem~\ref{th:equi_atomes}]

  Assume Point~\ref{th:item:atom}, that is, the set $A$ is an atom. By
  definition it has positive measure. By Lemma~\ref{lem:atom_convex},
  it is convex.  Since $A$ is a minimal admissible set with positive
  measure, we get Point~\ref{th:item:min_interval}.

Assume Point~\ref{th:item:min_interval}, that is, the set $A$ is minimal
convex with positive measure. It is irreducible thanks  to Lemma~\ref{lem:irr+cvx}.
As it is also  admissible (as a convex set), we  get Point~\ref{th:item:max_irre2}. 

Notice Point~\ref{th:item:max_irre2} implies Point~\ref{th:item:atom}
by Lemma~\ref{lem:irr+at}.

\medskip

Assume Point~\ref{th:item:atom} (and thus
Points~\ref{th:item:atom}-\ref{th:item:max_irre2} by the previous
proofs). So the set $A$ is irreducible.  Let us check it is maximal
irreducible.  Let $A' \supset A$ be another irreducible set.  As the
set $F(A)$ is $T$-invariant, we get that $F(A) \cap A'$ is
$T_{A'}$-invariant.  So by irreducibility of $A'$, we have
$F(A) \cap A' = A'$ as $A \subset F(A) \cap A'$ has positive measure.
We deduce that $A'\subset F(A)$, and similarly $A'\subset P(A)$.  This
gives $A'\subset F(A) \cap P(A)=A$ as $A$ is convex.  Therefore $A$ is
a maximal irreducible set, which proves Point~\ref{th:item:max_irre}.

\medskip

Assume Point~\ref{th:item:max_irre}, that is $A$ is 
a maximal  irreducible set. Thanks 
to Lemma~\ref{lem:irr_in_atom}, the set 
$P(A) \cap F(A)$ is an atom and thus irreducible by
Lemma~\ref{lem:atom_irr}. By
maximality of $A$, we have $A= P(A) \cap F(A)$, and thus $A$ is an
atom.  This  gives   Point
\ref{th:item:atom}.  
\end{proof}

\subsection{An intuitive order on atoms}
\label{sec:order}
Nelson \cite{nelson74}  introduced an  order relation on  atoms (therein
called  $k$-components, and  which  correspond  to  maximal irreducible  sets,
therefore to  atoms by Theorem~\ref{th:equi_atomes}) using the  past of
measurable  sets  (therein $k$-closures).  We  rewrite  this  order
relation, using futures instead of pasts for convenience.

\begin{defi}[Order relation between atoms]
  \label{defi:order}
  Let $T$ be a positive operator on $L^p$ with $p\in (1, +\infty
  )$. Let $A, B$ be two $T$-atoms. We denote $A \preccurlyeq B$ if
  $A \subset F(B) \muae$ (that is, if $F(A) \subset F(B) \muae$).

  We write $A\prec B$ when $A\preccurlyeq B$ and $A$, $B$ are not
  a.e.\ equal. 
\end{defi}

In the epidemiological interpretation of Remark~\ref{rem:epidem}, we have
$A \preccurlyeq  B$ if $A$ may  be infected by an  epidemics starting on
$B$.
We first give some equivalent definitions of this relation $\preccurlyeq$.
Recall $F^*(A)=F(A) \cap A^c$ and similarly for $P^*$. 

\begin{lem}[Equivalent definitions of $\preccurlyeq$]
  \label{lem:def_rel_ordre}
Let $A, B$ be two atoms such that $A$ and $B$ are not a.e.\ equal. The
following properties are equivalent. 

\begin{enumerate}[(i)]
	\item $A \subset F(B) \muae$. \label{prop:item:inc_fut}
	\item $A \subset F^*(B) \muae$. \label{prop:item:inc_fut_propre}
	\item $B \subset P(A) \muae$. \label{prop:item:inc_pas}
	\item $B \subset P^*(A) \muae$. \label{prop:item:inc_pas_propre}
\end{enumerate}
\end{lem}

\begin{proof}
  The equivalences between Points~\ref{prop:item:inc_fut}
  and~\ref{prop:item:inc_fut_propre} and between
  Points~\ref{prop:item:inc_pas} and~\ref{prop:item:inc_pas_propre}
  are direct consequences of the fact that two atoms are always
  equal$\muae$ or disjoint$\muae$. We also have that $A \subset F(B)$
  is equivalent to $A \cap F(B)\neq\emptyset$ as $A $ is an atom. By
  Lemma~\ref{lem:darknessoffuturepast}, as $B$ is also an atom, the
  property $A \cap F(B)\neq\emptyset$ is also equivalent to
  $B\subset P(A)$. This ends the proof.
\end{proof}

We can now check that this indeed defines an order relation.

\begin{prop}[$\preccurlyeq$ is an order relation]
  \label{prop:order}
The relation $\preccurlyeq$ is an order relation on the set of atoms.
\end{prop}

\begin{proof}
  The relation $\preccurlyeq$ is clearly reflexive and transitive by
  definition of $\preccurlyeq$ and by the monotony of the future, see
  Lemma~\ref{lem:elementary_past_future}~\ref{item:fut_monotone}.

  Let $A, B$ be two atoms such that $A \preccurlyeq B$ and
  $B \preccurlyeq A$. By definition $A\subset F(B)$, which implies
  $F(A)\subset F(B)$. A symmetry argument yields $F(B)\subset F(A)$,
  so that both are equal. Similarly $P(A)=P(B)$. Since $A$ and $B$ are
  convex, $A=P(A)\cap F(A) = P(B)\cap F(B) = B$, so
  relation $\preccurlyeq$ is an
  order relation.
\end{proof}

\subsection{Admissible/irreducible sets and atoms for $T$ and $T^n$}
\label{sec:admin-Tn}

We    end   this    section   with    some   comparison    between   the
admissible/irreducible  sets and  atoms of  $T$ and  $T^n$, with  $n\geq
2$.  We  denote by  $\ca(S)$  the  set  of $S$-admissible  sets, where
$S$ is a positive  operator. Let us point out that in the next lemma,
one can replace $T^n$ by $\expp{T}$ for example.

\begin{lem}[Admissible sets of $T^n$]\label{lem:adm_T^n}
  Let $T$ be  a positive operator on $L^p$ with  $p \in (1,+\infty)$ and
  $n \in \N^*$. 
\begin{enumerate}[(i)]
	\item \label{prop:item:adm_T^n}
          Any $T$-admissible set is  $T^n$-admissible, that is,
          $\ca(T)\subset \ca(T^n)$. 
	\item \label{prop:item:cvx_T^n}
          Any $T$-convex set is  $T^n$-convex.
    \item \label{prop:item:irre_T^n} If the operator $T^n$ is
      irreducible, then $T$ is irreducible. 
\end{enumerate}
\end{lem}

\begin{proof}
Lemma~\ref{lem:inv_T^n} gives  Point~\ref{prop:item:adm_T^n}. If a set $A$ is
$T$-convex, we deduce that $A=F(A)\cap P(A)$. Then use
Lemma~\ref{lem:inv_T^n} to deduce that $F(A)$ (resp. $P(A)$) is
$T^n$-invariant (resp. $T^n$-co-invariant) and then
Lemma~\ref{lem:equi_convex}~\ref{lem:item:inter_ic} to get that $A$ is
thus $T^n$-convex. Point~\ref{prop:item:irre_T^n} is immediate using Lemma~\ref{lem:inv_T^n}.
\end{proof}

We illustrate in  the next example that the operator  $T$ and its powers
may have different atoms.

\begin{ex}[Different atoms of $T$ and $T^2$]
   \label{ex:T-T2-atom}
We consider the finite state space  $\Omega = \{1, 2\}$ endowed
  with  the uniform  probability  $\mu$,  and the  kernel operator  $T_k$
  associated to the  kernel (or matrix as the space  is finite), given in
  Fig.~\ref{fig:m-graph12}. The operator $T_k$ has only one atom $\{1,2\}$,
  whereas its square $T_k^2$ admits two atoms $\{1\} $ and $\{2\}$.
  The fact that $\{1,2\}$ may be partitioned in $T^2$-atoms is in fact
  generic, see Proposition~\ref{prop:at_T^n} below. 
\end{ex}

\begin{figure}
\centering
\begin{subfigure}[b]{.4\textwidth}\centering
$\begin{pmatrix}
0 & 1 \\
1 & 0
\end{pmatrix}$
\caption{Matrix on $\Omega$.}
\label{fig:m-graph12}
\end{subfigure}
\begin{subfigure}[b]{.4\textwidth}\centering
\begin{tikzpicture}
\node[draw,circle](1) at (-1.5,0) {1};
\node[draw,circle](2) at (0,0) {2};
\draw[>=latex,->] (1)--(2);
\draw[>=latex,->] (2)--(1);
\end{tikzpicture}
\caption{Associated communication graph.}
\label{fig:g-graph12}
\end{subfigure}
\caption{Example of matrix and associated communication graph on
  $\Omega = \{1,2\}$ for which the atoms of the matrix and its square
  are distinct.
}
\label{fig:graph12}
\end{figure}
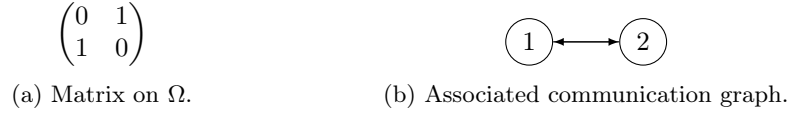

The admissible sets of $T$ and its power might differ even if there is
no atom. 

   \begin{ex}[No atoms and $\ca(T)\neq \ca(T^2)$]
  We continue Example~\ref{ex:adm_neq_bor}.  The operator $T_k^2$ is a kernel
  operator with a kernel $k^{\otimes 2}$ on $[0, 1]$, see
    Fig.~\ref{fig:kernel_2},  defined by:
 \begin{equation}
   \label{eq:def-k12}
    k^{\otimes 2}  (x,y)= (x-y) \left( \un_{\{y \leq x \leq
      1/2\}} + \un_{\{1/2 \leq y \leq x\}} \right).
\end{equation} 
The $T_k^2$-invariants  sets are a.e.\ equal  to $[ a, 1/2]\cup  [b, 1]$
with $a  \in [0,1/2]$  and $b\in [1/2,1]$,  whereas the  $T_k$ invariant
sets, see  Example~\ref{ex:adm_neq_bor}, corresponds to those  sets with
$b=a+1/2$.  Therefore the $\sigma$-field  of the $T_k^2$ admissible sets
is exactly  the Borel $\sigma$-field of  $[0 ,1]$; it does  not coincide
with  the  $\sigma$-field   of  the  $T_k$  admissible   sets  given  in
Example~\ref{ex:adm_neq_bor}.
\end{ex}

We now  check that the irreducible  sets of $T$ and  those of
$T^2$ are not always the same.

\begin{ex}[$T^2$-irreducibility does not imply $T$-irreducibility]
  \label{ex:irr-T2}
  We           consider           the           measured           space
  $(\Omega=[0,  1], \cf,  \leb)$, with  $\cf$ the  Borel
  subsets of $[0, 1]$ and $\leb$  the Lebesgue measure on $[0, 1]$,
  and     the    kernel     $k$     on    $[0,     1]$    defined
  by:
  \begin{equation}
   \label{eq:def-k3}
  k (x,y) =   \un_{\{x \leq 1/2 \leq  y\}} + \un_{\{y \leq 1/2 \leq x\}} \quad\text{(see Fig.~\ref{fig:kernel_3})}.
\end{equation}
 Then  the operator  $T_k^2$ is  a kernel
  operator     with     kernel      $k^{\otimes     2}$     given
  by:
\[
k^{\otimes  2} (x,y)=2^{-1}  \un_{\{\max(x,y)  \leq  1/2\}} +  2^{-1}
  \un_{\{\min(x,y) \geq  1/2\}}\quad\text{(see
    Fig.~\ref{fig:kernel_4})}.
\]
Then the  set $[0,1/2]$ is $T_k^2$-irreducible,  $T_k^2$-admissible (and
thus  a  $T^2_k$-atom),  and  $T^2_k$   invariant,  but  it  is  neither
$T_k$-irreducible  (as  $T_{[0,1/2]}  =  0$)  nor  $T_k$-admissible  (as
$[0, 1]$ is a $T_k$-atom).
\end{ex}

\begin{figure}
\centering
\begin{subfigure}[b]{.4\textwidth}\centering
\includegraphics[width=1\linewidth]{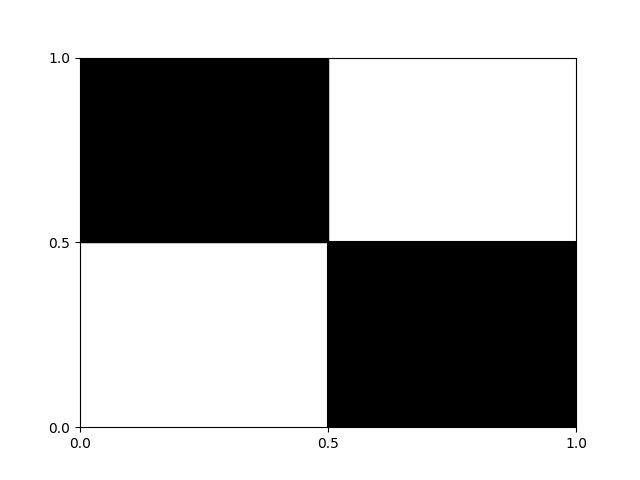}
\caption{Kernel $k$ defined in~\eqref{eq:def-k3}.}
\label{fig:kernel_3}
\end{subfigure}
\begin{subfigure}[b]{.4\textwidth}\centering
\includegraphics[width=1\linewidth]{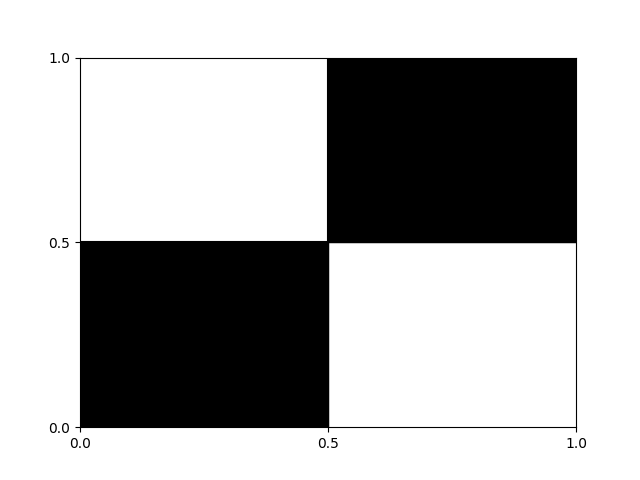}
\caption{Kernel $2\,k^{\otimes 2}$.}
\label{fig:kernel_4}
\end{subfigure}
\caption{Support of some $\{0, 1\}$-valued kernels.}
\label{fig:kernelbis}
\end{figure}

For $A\subset \Omega$ measurable and  $S$ be positive operators on $L^p$
with $p\in  (1, +\infty)$,  we denote  by $S(A)$  the support  (which is
defined a.e.)   of $Sf$,  where $f\in L^p$  is any  nonnegative function
whose  support is  a.e.\ equal  to $A$  (notice the  support of  $Sf$ is
defined  up  to  an  a.e.\   equivalence).   More  formally:  the  class
$\cp=\{B\in  \cf\,  \colon\,  k_S(B,A)=0\}$,   where  $k_S$  is  defined
in~\eqref{eq:def-kT},    is   stable    by    countable   union;    thus
Lemma~\ref{lem:min-exist}  implies the  existence of  a maximal  set for
$\cp$; then by definition its complementary  is equal to $S(A)$.  We now
state some corresponding preliminary properties in the next two lemmas.

\begin{lem}[Basic properties of $T(A)$]
  \label{lem:T^kA}
  Let $T,S$ be positive operators on $L^p$ with $p\in (1, +\infty)$, and
    $A$  a measurable set. We have the following properties.
  \begin{enumerate}[(i)]
	\item \label{item:TA_support_f} 
          $\supp(T(f)) = T(\supp(f))$ a.e. for any $f \in L^p_+$. In particular, if $\un_A$
          belongs to $L^p$, then we have $T(A) = \supp(T(\un_A))$ a.e..
	\item \label{item:TSA}
          $T(S(A))=(TS) (A)$ a.e. and $(T+S)(A)=T(A) \cup S(A)$ a.e.. 
	\item\label{item:TAinTB} If $A \subset B$ a.e., with $B$ a
          measurable set, then we have $T(A) \subset
     T(B)$ a.e..
   \item \label{item:TA+B} Let $(A_i)_{i\in I}$  be an at most countable
     family of measurable sets.  We have:
   $$T\left( \bigcup\limits_{i \in I} A_i \right) = \bigcup\limits_{i \in I} T(A_i) \text{ a.e.} \quad \text{and} \quad T\left( \bigcap\limits_{i \in I} A_i \right) \subset \bigcap\limits_{i \in I} T(A_i) \text{ a.e.}.$$ 
  \end{enumerate}
\end{lem}

\begin{proof}
Let $f' \in L^p$ such that $f' > 0$ and $\un_{\supp(f)} f' = f$. Then, by~\eqref{eq:kT=0}, we have for any measurable set $B$ that $k_T(B, \supp(f)) = 0$ if and only if 
$B \cap \supp\left(T\left(\un_{\supp(f)} f'\right)\right) =
\emptyset$. This gives  Point~\ref{item:TA_support_f}.
Point~\ref{item:TSA} is a direct consequence of Point~\ref{item:TA_support_f} applied to $f \un_A$ for any positive function $f \in L^p$.
Point~\ref{item:TAinTB} is a direct
   consequence of the positivity of $T$.

\medskip

We now prove Point~\ref{item:TA+B}.  Let $B$ be a measurable set. As the
map   $k_T(B,   .)$   is   non-decreasing   and   $\sigma$-additive   on
$\mathcal{F}$, we have $k_T \left( B, \bigcup_{i \in I} A_i \right) = 0$
if and only if  for all $i \in I$, we have $k_T(B,  A_i) = 0$.  Thus the
maximal            set           $B$            that           satisfies
$k_T   \left(   B,   \bigcup_{i   \in   I}   A_i   \right)   =   0$   is
$\bigcap_{i        \in        I}       T(A_i)^c$,        that        is,
$T\left( \bigcup_{i \in I} A_i \right) = \bigcup_{i \in I} T(A_i)$.  The
property
$T\left( \bigcap_{i \in I} A_i \right) \subset \bigcap_{i \in I} T(A_i)$
is  a  direct  consequence  of  Point~\ref{item:TAinTB}.  We  thus  have
Point~\ref{item:TA+B}.
\end{proof}

\begin{lem}[$T^k(A)$ and invariance/irreducibility]
  \label{lem:Tk(A)_inv}
  Let $T$ be a positive operator on $L^p$ with $p\in (1, +\infty )$. Let
  $A$ be a measurable set, and $n \in \N^*$. We have the following properties.
  \begin{enumerate}[(i)]
      \item\label{item:TAinA}
    The set $A$ is $T$-invariant if and only if $T(A) \subset A$ a.e..
    \item If the set $A$ is $T^n$-invariant, then for all $k \in \N$, the set $T^k(A)$ is $T^n$-invariant. \label{item:Tk_Tn_inv}
  \item If $T$ is a non-zero irreducible operator and $\mu(A) > 0$, then
    we have $\mu(T(A)) > 0$. Moreover, we have $T(\Omega) = \Omega$
    a.e..
    \label{item:T_irre}
\end{enumerate}
\end{lem}

\begin{proof}
  By  definition   the  set  $A$   is  $T$-invariant  if  and   only  if
  $A^c \cap  TA=\emptyset$; this gives Point~\ref{item:TAinA}.   Let the
  set   $A$   be    $T^n$-invariant   and   $k   \in    \N$.   Then   by
  Lemma~\ref{lem:T^kA}~\ref{item:TSA},              we              have
  $T^n(T^k(A)) =  T^k(T^n(A))$.  Since  we have  $T^n(A) \subset  A$, we
  deduce    that    $T^n(T^k(A))    \subset    T^k(A)$.    This    gives
  Point~\ref{item:Tk_Tn_inv}.
   
\medskip

Assume that $T$ is a non-zero irreducible operator and that  $\mu(T(A)) = 0$.
The latter condition implies that  $A$ is $T$-invariant, and by
irreducibility of $T$, that  $A = \emptyset$ or $A = \Omega$.
As $T$ is a non-zero operator, we get the latter case is impossible and
thus we have $\mu(A)=0$. 
As the set $T(\Omega)$ is $T$-invariant with positive measure, we deduce
that  $T(\Omega) = \Omega$ by the previous argument. This gives Point~\ref{item:T_irre}.
\end{proof}

The following corollary provides an interesting
link between the future of a set and the exponential of $T$. 
\begin{cor}[Future and $\expp{T}$]
  \label{cor:expT}
  Let  $T$ be a positive operator on $L^p$ with $p \in (1, +\infty)$ and 
  $A$  a measurable set. We have:
  \[
    \expp{T}(A)= \bigcup_{n \in \N} T^n(A)=F(A)\quad\text{a.e.}.
  \]
\end{cor}

\begin{proof}
  The first equality is elementary, using the same arguments as for
  Lemma~\ref{lem:T^kA}~\ref{item:TSA}.  We prove the second
  equality.  The set $\bigcup_{n \in \N} T^n(A)$ is clearly
  $T$-invariant by Lemma~\ref{lem:Tk(A)_inv}~\ref{item:TAinA} and
  contains $A$, we therefore have
  $F(A) \subset \bigcup_{n \in \N} T^n(A)$.  As $F(A)$ is a
  $T$-invariant set, it is a $T^n$-invariant set for any $n \in \N$ by
  Lemma~\ref{lem:inv_T^n}.  We get that for any $n \in \N$,
  $T^n(F(A)) \subset F(A)$ by
  Lemma~\ref{lem:Tk(A)_inv}~\ref{item:TAinA}, and thus
  $\bigcup_{n \in \N} T^n(A) \subset \bigcup_{n \in \N} T^n(F(A))
  \subset F(A)$. This gives the second equality.
\end{proof}

We give the following result on the restriction of $T^n$ on a convex set.

\begin{lem}[Power of a restricted operator on a convex set]
  \label{lem:pow_rest}
Let  $T$ be a positive operator on $L^p$ for $p \in (1, +\infty)$ and $A$ a convex set. Then we have $(T_A)^n = (T^n)_A$  for any $n \in \N^*$.
\end{lem}

For $n\in \N^*$, we will thus use the notation $T^n_A$ for $(T_A)^n =
(T^n)_A$ when  $A$ is a convex set. 

\begin{proof}
Let $n\in \N ^*$. We have:
  \[
    (T^n)_A=M_A T^n M_A
    =M_A T^{n-1} M_A T M_A + M_A T^{n-1} M_{F^*(A)} T M_A
    = (T^{n-1})_A \, T_A,
  \]
  where we used  that $T(A) \subset F(A)= A \cup  F^*(A)$ for the second
  equality, and  that $F^*(A)$  is $T$-invariant (as  $A$ is  convex, see
  Lemma~\ref{lem:equi_convex})   and  thus   $T^{n-1}$-invariant,  so   that
  $M_A T^{n-1} M_{F^*(A)}=0$ for the last.  We conclude by iteration.
\end{proof}

The following  result on the decomposition  of atoms is also  related to
\cite[Theorem~8]{schwartz_61} which  states that the eigenvalues  of $T$
(when $T$ is compact) whose modulus  are equal to the spectral radius of
$T$  are roots  of  unity.  We  say  that a  family  of measurable  sets
$(A_i)_{i \in  I}$ forms an  \emph{a.e.\ partition} of a  measurable set
$B$   if   we   have: 
$A_i  \cap   A_j  =   \emptyset$  a.e.\   for  any   $i  \neq   j$,  and
$B = \bigcup_{i \in I} A_i$ a.e..

\begin{prop}[Atoms of powers of $T$]\label{prop:at_T^n}
Let $T$ be a positive operator on $L^p$ with $p\in (1, +\infty)$ and $n \in \N^*$.
We have the following properties.
\begin{enumerate}[(i)]
	\item \label{item:Tn-to-Tatom} If $A$ is a $T^n$-atom, then
          there exists a $T$-atom $B$ such that $A\subset B$. 
	\item   \label{item:from_irre_to_non_irre}   Let    $B$   be   a
          $T$-atom. There exists a $T^n$-atom $A\subset B$ and a divisor
          $\rd$ of  $n$ such  that the family  $(A_k)_{0 \leq  k \leq  \rd-1}$, where
          $A_k=T^k  (A)\cap  B$,  forms  an$\muae$ partition  of  $A$  in
          $T^n$-atoms.
\end{enumerate}
\end{prop}

The second point is slightly more technical; its proof is given in the
next section. 

\begin{proof}[Proof of Point~\ref{item:Tn-to-Tatom}]
  Let $A$ be a $T^n$ atom.  The family
  $\cp=\{B\in \ca(T)\, \colon\, A\subset B\}$ of measurable sets is
  clearly stable by countable intersection.  Let $A'$ denote a minimal
  set for $\cp$, given by Lemma~\ref{lem:min-exist}.  Let
  $B\in \ca(T)$ such that $B\subset A'$.  As $B\in \ca(T^n)$ by
  Lemma~\ref{lem:adm_T^n}~\ref{prop:item:adm_T^n}, we get that either
  $A\subset B$ or $A\cap B=\emptyset$. By the minimality of $A'$, we
  deduce in the former case that $A'=B$ and in the latter case that
  $A' \cap B=\emptyset$, and thus $B = \emptyset$.  This gives that
  $A'$ is a $T$-atom which contains $A$.
\end{proof}

\subsection{Proof of
  Proposition~\ref{prop:at_T^n}~\ref{item:from_irre_to_non_irre}} 

Thanks to  Lemma~\ref{lem:pow_rest} (with  $A$ replaced  by $B$),  it is
enough to consider the  case where $\Omega$ is a $T$-atom,  that is, $T$ is
irreducible.  The case  $T=0$ being  trivial,  we shall  assume in  this
section only  that $T$ is a  positive irreducible operator on  $L^p$ for
$p  \in   (1,  +\infty)$  and   $T\neq  0$.   In  particular,   we  have
$T(\Omega)=\Omega$ a.e.\ (see
Lemma~\ref{lem:Tk(A)_inv}~\ref{item:T_irre})  and $F(A)=\Omega$ a.e.\  for any  measurable set  $A$ with
positive measure.  Motivated by  Corollary~\ref{cor:expT}, we define, for
any measurable set~$A$ with positive measure, the quantity:
\[
  n_A= \inf\left\{m \in \N^*\cup\{\infty \}\, \colon\,  \bigcup_{j =
      0}^{m-1} T^j(A)= \Omega \quad \text{a.e.}\right\}.
  \]
  If $A$ is a $T^n$ invariant set  with
  positive measure, the set $\bigcup_{j=0}^{n-1} T^j A$ is $T$-invariant
  and contains $A$; by irreducibility it must be equal to $\Omega$,
  so  $n_A\leq  n$. It is also elementary to check that if
  $A\subset B$ a.e.\ for a measurable set $B$, then   $n_A\geq n_B\geq 1$.
  
Let $\ci_n^*$ be  the family of $T^n$-invariant sets with
positive measure.  This set is  non empty  as it contains  $\Omega$, and
we have $n\geq  n_A\geq 1$ for all $A\in \ci^*_n$.
We have the following technical properties. 
  
\begin{lem}[Elementary properties]
  \label{lem:elementary-nA}
  Let   $n\in
  \N^*$ and $A\in \ci^*_n$ (\emph{i.e.}, a non trivial $T^n$-invariant
  set). 
\begin{enumerate}[(i)]
 \item \label{item:nA>k}
 Let $\ell\in \N$. We have for $k\in \N^*$:
   \[
     \bigcup _{j=\ell}^{k+\ell-1}T^j( A)=\Omega \quad\text{a.e.}
     \quad\Longleftrightarrow\quad
     n_A\leq k.
   \]
In particular, we have $n_{T^\ell(A)}=n_A$. 

 \item  \label{item:nB>nA}  Set
   $B=A \bigcap \left(\bigcup  _{j=1}^{n_A-1}  T^j(A)\right)$  (notice   the  indices  $j$  are
   positive). We have:
\[
  \mu( B)>0
  \quad\Longrightarrow
  \quad
  n_B>n_A.
\]
\end{enumerate}
\end{lem}

\begin{proof}
We prove Point~\ref{item:nA>k}. 
The  set $B=\bigcup _{j=0}^{k-1} T^j(A)$
is   $T^n$-invariant    as   union   of   $T^n$-invariant    sets,   see
Lemma~\ref{lem:Tk(A)_inv}~\ref{item:Tk_Tn_inv},         and         thus
$T^n(B)   \subset  B$.   If   $T^\ell ( B)=\Omega$,   then   we  get,   as
$(\ell+1)n -\ell\geq 0$ and $T(\Omega)=\Omega$:
\[
  \Omega=T^{(\ell+1)n -\ell} (\Omega)=T^{(\ell+1)n}(B) \subset B,
\]
and thus $B=\Omega$ and $n_A\leq  k$. 
On the other hand, if $n_A\leq  k$, then we have $B=\Omega$ and $T^\ell
(B)=\Omega$.
\medskip

We         prove         Point~\ref{item:nB>nA}.         The         set
$B=A    \bigcap   \left(\bigcup    _{j=1}^{n_A-1}   T^j(A)\right)$    is
$T^n$-invariant,  and thus  belongs to  $\ci^*_n$ as  $\mu(B)>0$.  Using
$B\subset A$ and  thus $T^j(B) \subset T^j(A)$ for all  the terms $j\geq 0$,
we get:
\[
  \bigcup _{j=0}^{n_A-1} T^j(B)\subset  \bigcup _{j=1}^{n_A-1} T^j(A). 
\]
By Point~\ref{item:nA>k} (with $\ell=1$), the latter set is not a.e.\ equal to $\Omega$,
which in turns,  using Point~\ref{item:nA>k} again (but with $\ell=0$), implies that
$n_B>n_A$. 
\end{proof}

Let        $n\geq        2$.        The            supremum
$ n_{\max}=\sup\left\{n _A\, \colon\, A\in  \ci_n^*\right\} $ is less or
equal than $n$ and is thus a maximum.

We can directly deduce Proposition~\ref{prop:at_T^n}~\ref{item:from_irre_to_non_irre}
from the next lemma.

\begin{lem}
  \label{lem:mx-nA}
Let $A$ be a $T^n$-invariant set with positive measure  such that
$n_A=n_{\max}$. We have, with  
$A_k=T^k(A)$ for $k\in \N$:
\begin{enumerate}[(i)]
\item\label{item:nA|n}
  $n_A$ is a divisor of $n$. 
\item\label{item:periodique}
  $T^{n_A}(A_k)=A_k$ a.e. for all $k\in \N$.
\item\label{item:disjoint}
  $A_k\cap A_\ell=\emptyset$ a.e.\ for all $k\neq \ell$ in $\{0,
    \ldots, n_A -1 \}$. 
  \item\label{item:Ak=atom}
    The sets $(A_k)_{k\in \{0, \ldots n_A-1\}}$ are 
  $T^n$-atoms.
\end{enumerate}
\end{lem}

\begin{proof}
  Let $A$ be $T^n$-invariant such that $n_A=n_{\max}$.
   Set $A^*_k=\bigcup _{j\in \{0, \ldots, n_A-1\} \setminus \{k\}} A_j$ for
   $k\in \{0, \ldots, n_A-1\}$ (so that $A_k \cup A_k^*=\Omega$ by
   definition of $n_A$) and $B=A \cap A^*_0$. The set $B$ is
   invariant.  We assume that
   $\mu(B)>0$.  Since $B\subset A$, we
   get $n_B\geq  n_A$ and thus $n_B=n_A$ by maximality of $n_A$. Then,
   Lemma~\ref{lem:elementary-nA}~\ref{item:nB>nA} implies that
   $\mu(B)=0$.  By contradiction, we deduce  that $\mu(B)=0$, that is:
   \[
     A \cap A_0^*=\emptyset.
   \]
   Using that $T(\Omega)=\Omega$ as $T$ is irreducible, we get:
   \[
  A  \sqcup A^*_0=  \Omega=T(\Omega)= T^{n_A}(A) \cup A^*_0.
\]
This implies that $A \subset T^{n_A}(A)$. 
Writing $n=kn_A +r$ with $r\in \{0, \ldots, n_A-1\}$, we get:
\[
  T^r(A)\subset T^{r+ n_A}(A)\subset T^{r+ k n_A}(A)=  T^n (A) \subset A.
\]
If  $r>0$, this  would imply  that $n_A\leq  r$. As  $r<n_A$, we  thus
deduce  that  $r=0$,  that   is  Point~\ref{item:nA|n},  and  then  that
$A=T^{n_A}(A)$.  This gives  Point~\ref{item:periodique}  for $k=0$  and
thus for  any $k$, as the  $T^n$-invariant set $A_k$ is  also maximal in
the         sense         that         $n_{A_k}=n_A=n_{\max}$         by
Lemma~\ref{lem:elementary-nA}~\ref{item:nA>k}.

Using again that $A_k$ is maximal and that $T^{n_A}(A_j)=A_j$, we can
apply the previous argument to get that $A_k \cap A^*_k=\emptyset$ for
all $k\in \{0, \ldots, n_A-1\}$. This readily implies that the $A_k$ for 
$k\in \{0, \ldots, n_A-1\}$ are pairwise disjoint, that is,
Point~\ref{item:disjoint}.

To conclude, it is enough to check Point~\ref{item:Ak=atom} for $k=0$.
As $A$ is $T^n$-invariant, to prove it is a $T^n$-atom, it is enough to
check that 
if  $B\subset A$ is a $T^n$-invariant set  with positive measure, then
$B=A$. Consider such a set $B$. Notice that $n_B$ is finite (as $B\in
\ci^*_n$) and that $n_B\geq  n_A$, that is $n_B=n_A$ by maximality of
$n_A$. We thus have:
\[
  A  \sqcup A^*_0= \Omega=B \bigcup \left(\bigcup  _{j=1}^{n_A-1}  T^j(B)\right)
\,\subset\, B \cup A^*_0.
\]
This readily implies that $A\subset B$ and thus $B=A$. 
\end{proof}

\section{Atoms and nonnegative eigenfunctions}\label{sec:eigenfunctions}

Until the  end of  this section,  $T$ is  a power compact  (that is,  there exists
$k \in \N^*$ such that the  operator $T^k$ is compact) positive operator
on $L^p$, where $p \in (1, +\infty)$ and $(\Omega, \mathcal{F}, \mu)$ is
a measured space  with $\mu$ $\sigma$-finite and  non-zero.  The purpose
of this section is to study  the intricate links between the ordered set
of atoms  and spectral  properties of $T$.   Especially, we  study links
between atoms  and nonnegative  eigenfunctions of  $T$. We  also provide
some criteria of monatomicity of  $T$.  The power compactness hypothesis
opens access to  different results, giving the  existence and uniqueness
under  irreducibility  of  nonnegative  eigenfunctions  for  a  positive
operator.

\subsection{On positive power compact operators}
\label{sec:pcp-operator}
Recall that $\rho(T)$ defined  in~\eqref{eq:def-rho} denote the spectral
radius  of  the  operator  $T$.
The algebraic multiplicity of $\lambda\in \C$ of $T$ is defined by:
\begin{equation}
   \label{eq:vp-mult}
\mult(\lambda, T) = \dim \left( \bigcup\limits_{k \in \N^*} \ker (T -
  \lambda \id)^k \right).
\end{equation}
The  complex  number  $\lambda\in  \C$  is an  eigenvalue  of  $T$  when
$\mult(\lambda,  T)\geq 1$,  it  is simple  when $\mult(\lambda,  T)=1$.
When  $T$ is  power  compact, the  multiplicity  $\mult(\lambda, T)$  is
finite for $\lambda\in \C^*$,  see \cite[Theorem p.~21]{konig86}.  Notice
that for power  compact operators the multiplicity  of $\lambda\in \C^*$
is also the dimension of the  range of the spectral projection (which is
the definition  used in \cite{schwartz_61} and  \cite{dunford88}) thanks
to \cite[Theorems VII.4.5-6]{dunford88}.

For a measurable set $A\subset \Omega$, when there is no ambiguity on
the operator $T$, we simply write $\rho(A)= \rho(T_A)$, see
Section~\ref{sec:leb}, and $\mult(\lambda, A)=\mult(\lambda, T_A)$ for
the spectral radius and multiplicity of $\lambda$ for
$T_A= M_A T M_A$, see~\eqref{eq:def-TA}, the operator $T$ restricted to
$A$.

The following lemma proves that the restriction of a power compact operator is also power compact.

\begin{lem}[Restriction of a power compact operator]\label{lem:rest_pow_cpct}
Let $T$ be a positive power compact operator on $L^p$. Then there exists $k \in \N^*$ such that for any measurable set $\Omega'$, the operator $(T_{\Omega'})^k$ is compact.
\end{lem}

\begin{proof}
  Let   $n\in   \N^*$   such   that    $T^n$   is   compact.   We   have
  $ 0 \leq (T_{\Omega'})^n \leq T^n$. Since
  $T^n$ is compact, we get thanks to
  \cite[Theorem 5.13]{aliprantis06} that $(T_{\Omega'})^{3n}$ is compact.
\end{proof}

We say that the atom  $A\subset \Omega$  is \emph{non-zero}  if
$\rho(A)>0$, and denote by $\anz$ be the  (at most countable) set 
 of  non-zero   atoms:
\begin{equation}
   \label{eq:def-anz}
  \anz=\{A\in \atom\, \colon\, \rho(A)>0\}.
\end{equation}
   Notice  that
$\mult(\lambda,   A)=0$   for   all   atoms   $A\in \atom \setminus   \anz$   and
$\lambda\in \C^*$.

 We recall in our framework the  classical results related to power
 compact operators.

\begin{theo}
  \label{theo:rappel}
Let $T$ be a positive power compact operator on $L^p$ with $p \in (1,+\infty)$. 
\begin{enumerate}[(i)]
\item \textbf{Krein-Rutman.} If $\rho(T)$  is positive then $\rho(T)$ is
  an  eigenvalue of
  $T$, and there exists  a corresponding nonnegative right eigenfunction
  denoted 
  $v_T$.\label{th:item:KR}
  \item \textbf{de Pagter.} If $T$ is irreducible then $\rho(T)$ is positive unless $T=0$
    and $\dim(L^p)=1$, that is, if $A$ is measurable then either
    $\mu(A)=0$ or $\mu(A^c)=0$.\label{th:item:dP}
  \item \textbf{Perron-Jentzsch.} If $T$ is irreducible with
    $\rho(T)>0$, then $\rho(T)$ is simple, 
      $v_T$ is positive a.e., and $v_T$ is the unique  nonnegative right
      eigenfunction  of $T$.\label{th:item:PJ}
    \item \textbf{Schwartz.}
We have for $\lambda\in \C^*$:
\begin{equation}
  \label{eq:mult}
  \mult(\lambda, T) = \sum_{A\in \anz} \mult( \lambda, A)
  \quad\text{and}\quad
  \rho(T)=\max _{A\in \anz} \rho(A). 
\end{equation}\label{th:item:S}
    \end{enumerate}
  \end{theo}
  \begin{rem}
    In the  Perron-Jentzsch result  and in  what follows,  uniqueness of
    eigenfunctions is understood up to a multiplicative constant.
  \end{rem}
  
\begin{proof}
  We first recall the vocabulary used by Grobler \cite{grobler87}. For
  any $v \in L^p$, we denote by  $E_v$ the smallest band (therefore the
  smallest subspace of the form $L^p_A$ with $A \in \mathcal{F}$)
  that contains $v$, that is $L^p_{\supp(v)}$. We say that $v \in L^p$
  is quasi-interior if the closure of $E_v$ is equal to $L^p$, that is
  if $v > 0 \muae$.

\medskip

Point~\ref{th:item:KR}  is  given  by \cite[Theorem  3]{grobler87},  and
Point~\ref{th:item:dP}  by \cite[Theorem  12 (1)]{grobler87}.   To prove
Point~\ref{th:item:PJ},  by \cite[Theorem~12~(1)]{grobler87},  since $T$
is irreducible, $\rho(T)$  is a simple eigenvalue  and the corresponding
eigenfunction is  a quasi-interior  point of $L^p$,  that is  a positive
eigenfunction.  By \cite[Theorem 5.2 (iv), p.~329]{schaefer_74} (that
can be applied as $T$ is  power compact, see Corollary p.~329), $\rho(T)$
is  the  only eigenvalue  related  to  a nonnegative  eigenfunction.  As
$\rho(T)$ is  simple, $v_T$ is  the unique nonnegative  eigenfunction of
$T$.

\medskip

Point~\ref{th:item:S} is  an extension  of~\cite[Theorem 7]{schwartz_61}
(stated  for $\mu$  finite  and  $T$ compact),  and  its  proof is  very
similar. We  provide a short  proof for  completeness.  Let $h  \in L^1$
with  $1\geq  h  >  0  \muae$; thus  the  measure  $h.\mu$,  defined  by
$h.\mu(A) = \int_A h(s) \mu(\rd s)$  for $A \in \mathcal{F}$, is finite.
Following the  proof of~\cite[Theorem  7]{schwartz_61}, it is  enough to
check that Lemmas 4,  11 and 12 therein also hold  by replacing $\mu$ by
$h.\mu$ in their statement and when the operator $T$ is power compact.

For Lemma 11,  the proof given by \cite{schwartz_61} is  also valid when
the  operator $V$  given therein  is power  compact, as  every point  of
$\spec (V)  \priv{0}$ is isolated and  as for any $\lambda  \neq 0$, the
quantity    $\mult(\lambda,   V)$    is   finite,    see   \cite[Section
VII.4]{dunford88}.  For Lemma 12,  the proof given by \cite{schwartz_61}
holds for any positive operator, and also  holds when we replace $\mu$ in the
statement by the finite measure $h.\mu$.

Lemma 4  states that if  $\mu$ is finite and  $T$ is a  positive compact
operator, then for all $\lambda > 0$ there exists $\delta > 0$ such that
for all measurable set  $A \in \cf$ such that $\mu(A)  < \delta$ we have
$\rho(T_A) <\lambda$. An elementary adaptation  of the proof of Lemma 4,
gives that the result also holds if $\mu$ is $\sigma$-finite provided we
replace the condition $\mu(A) < \delta $ by $h.\mu(A) < \delta$.  We now
assume that the operator  $T$ is power compact, and let  $k \in \N^*$ be
such  that the  operator $T^k$  is compact.   For $\lambda  > 0$,  there
exists $\delta >  0$ such that for  all measurable set $A  \in \cf$ with
$h.\mu(A)   <  \delta$   we  have   $\rho((T^k)_A)<  \lambda^k$.   Since
$0\leq     (T_A)^k     \leq     (T^k)_A    $,     we     deduce     that
$\rho((T_A)^k)\leq      \rho((T^k)_A)<      \lambda^k$,     that      is
$\rho(T_A)  < \lambda$  thanks  to  \cite[Theorem p.~21]{konig86}.   This
readily gives  the extension of  Lemma 4 to $\mu$  $\sigma$-finite and
$T$   positive    power compact.   This    concludes   the    proof   of
Point~\ref{th:item:S}.
\end{proof}

Let us stress that Theorem~\ref{theo:rappel} also applies to
$T^\star$. Indeed, the operator $T$ is irreducible (resp. positive,
resp. power compact) if and only if the operator $T^\star$ is
irreducible (resp. positive, resp. power compact).  By \cite[Theorem
p.~21]{konig86}, when $T$ is power compact, we have $\rho(T^\star)=\rho(T)$ as
well as $\mult(\lambda, T^\star)=\mult(\lambda, T)$ for all
$\lambda\in \C^*$.

The     following    result     is    a     direct    consequence     of
Theorem~\ref{theo:rappel},    as   any    atom    is   irreducible    by
Theorem~\ref{th:equi_atomes}.   The function  $v_A$ below will  be called  the
Perron-like eigenfunction of $T_A$.

\begin{cor}[Perron-like eigenfunctions for $T_A$]
  \label{cor:rappel}
  Let $T$ be a positive power compact operator  on $L^p$ with $p \in (1, +\infty)$
  and  $A$  a  non-zero  atom.   Then $\rho(A)$  is  a  simple  positive
  eigenvalue  of  $T_A$ and  there  exists  a unique  nonnegative  right
  eigenfunction of  $T_A$, say  $v_A$; furthermore  its support  is $A$,
  that is, $\supp(v_A)=A$ a.e., and we have $\rho(v_A)=\rho(A)$: $
  T_A v_A=\rho(A) v_A$.
\end{cor}

For $\lambda>0$, let  $\atom(\lambda)$ be the set of atoms with spectral
radius $\lambda$:
\begin{equation}
  \label{eq:def-atom(l)}
  \atom(\lambda)=\{ A\in \anz\, \colon\, \rho(A)=\lambda\}.
\end{equation}
We have the following elementary result,  with  the convention  $\max \emptyset = 0$. 
\begin{lem}[Spectral radius of restricted operators]
  \label{lem:prop-rho}
Let $T$ be a positive power compact operator on $L^p$ with $p \in (1,
+\infty)$.
\begin{enumerate}[(i)]
\item \label{item:prop-rho_fini}
  For  any $\lambda > 0$,  there exists a
  finite  number  of  atoms  with a  spectral  radius larger  than
  $\lambda$.
\item \label{item:prop-rho_O'}
  If $\Omega'$ is admissible, then we have:
 \begin{equation}
   \label{eq:rho=lax}
   \rho(\Omega')=\max_{A\in \anz, \, A\subset \Omega'} \rho(A).
 \end{equation} 
 \item\label{item:prop-atom(rho)}
   If $\rho(T)$ is positive, then we have
   $\mult(\rho(T), T)=\card(\atom(\rho(T))$. 
\end{enumerate}
\end{lem}
\begin{proof}
  By Corollary~\ref{cor:rappel}, any atom with a spectral radius
  $\rho(A) > 0$ satisfies $\mult(\rho(A), A) = 1$. If $\lambda$ is
  positive, then by \cite{dunford88}, the set
  \( \{ z\in \C, |z| \geq \lambda, \mult(z,T) \neq 0\}\)
  is finite (notice that $\mult(z, T) \in \N$ by
  \cite[Theorem p.~21]{konig86}). Therefore, by
  Theorem~\ref{theo:rappel}~\ref{th:item:S}, only a finite number of
  atoms $A$ may satisfy $\rho(A) \geq \lambda$, that is
  Point~\ref{item:prop-rho_fini}.
  Point~\ref{item:prop-rho_O'} then follows from~\eqref{eq:mult}, since the atoms
  of $T_{\Omega'}$ are precisely the atoms of $T$ that are
  included in $\Omega'$, by Proposition~\ref{prop:T-atom}~\ref{item:T'a-Ta}.

Finally, for any atom $A$, we have $\rho(A) \leq \rho(T)$, therefore the only
atoms with $\mult(\rho(T), A) > 0$ are exactly those with
$\rho(A) = \rho(T)$. By Corollary~\ref{cor:rappel}, these atoms
satisfy $\mult(\lambda, A) = 1$, thus we deduce 
Point~\ref{item:prop-atom(rho)} from~\eqref{eq:mult}.
\end{proof}

We directly deduce
from~\ref{item:prop-rho_O'} the following result.

\begin{lem}[The operator is quasi-nilpotent outside the non-zero atoms]
  \label{lem:quasi-nilp}
  The restriction $T_{\Omega'}$ of $T$ to $\Omega'$, the complement set 
  of  $\bigcup_{A\in \anz} A$, is quasi-nilpotent, that is,  
  $\rho(\Omega')=0$.
  \end{lem}

  \subsection{Nonnegative eigenfunctions}\label{sec:nonneg_eigen}
  The goal of this section is to describe exactly the set of nonnegative
  eigenfunctions and prove  Theorem~\ref{thI:carac_nonneg_vp}.  We start
  by two elementary results.

  \begin{lem}
    \label{lem:T_equal_TC}
    If $C$ is convex, and $\supp(v)\subset F(C)$, then we have 
  $T_Cv = \un_C Tv$.  
  \end{lem}
  \begin{proof}
    Since $C$ is convex, $F(C) = C\sqcup F^*(C)$ where $F^*(C)$ is invariant
    by Lemma~\ref{lem:equi_convex}. Since $\supp(v)\subset F(C)$,
 we have    $v = v\un_C + v\un_{F^*(C)}$. The statement follows by checking that, by
    Lemma~\ref{lem:inv_sets_ideals}, $\un_C T(v\un_{F^*(C)}) = 0$.
  \end{proof}
  \begin{lem}[Nonnegative eigenfunctions on an atom]
    \label{lem:fct_supp_F(A)}
Let $T$ be a positive operator on $L^p$ for $p \in (1, + \infty)$ and $A$ a non-zero atom.
If $v$ is a nonnegative right eigenfunction with $A\subset \supp(v) \subset F(A)$,
then $v$ coincides on $A$ with the Perron like right eigenfunction: 
 $\un_A v = \cste v_A$ for some $\cste > 0$, and $\rho(v)=\rho(A)$,
 that is,  $T v = \rho(A) v$.
\end{lem}

\begin{proof}
  Let $\lambda \geq 0$ with
  $T v = \lambda v$.
  Since $\supp(v) \subset F(A)$, we may apply Lemma~\ref{lem:T_equal_TC}
  to the atom $A$, which is convex by~Theorem~\ref{th:equi_atomes}, to get
\( T_A(\un_A v) = T_A v=\un_A Tv = \lambda \un_A v,\)
that is, $\un_A v$ is a
  nonnegative eigenfunction of $T_A$. Since $A\subset \supp(v)$, we get 
  $\un_A v$ is non-zero. By Corollary~\ref{cor:rappel}, we have $\lambda
  = \rho(A)$ and 
  $\un_A v = \cste v_A$ for some $\cste > 0$, as claimed. 
\end{proof}

We need an adaptation of~\cite[Theorem~4]{nelson74},
a result originally stated for kernel operators, and
which concerns subsolutions to the eigenvalue equation,
  that is, functions $f$ that satisfy:
  \begin{equation}
    \label{eq:subsolution}
    Tf \leq \lambda f.
    \end{equation}
\begin{prop}[\textbf{Nelson}: Nonnegative subsolutions are Perron eigenfunctions]
  \label{prop:if_sub_then_eigen}
  Let  $T$ be  a positive power compact irreducible operator  on $L^p$  with
  $p   \in    (1,   +\infty)$. If $f\in L^p_+$ satisfies \eqref{eq:subsolution}
  for some $\lambda\in(0,\rho(T)]$, then we have $Tf = \rho(T) f$.
\end{prop}
\begin{proof}
  Let $f \in L^p_+$ be a solution of~\eqref{eq:subsolution}. Without
  loss of generality we may assume $\lambda = \rho(T)$.  By the
  Perron-Jentzch theorem (Theorem~\ref{theo:rappel}~\ref{th:item:PJ}),
  there exists a nonnegative left eigenfunction $h \in L^q_+$ with
  left eigenvalue $\rho(T)$ such that $h > 0$ a.e.. Taking the bracket
  of~\eqref{eq:subsolution} with the nonnegative function $h$,
  and
  using the fact that it is a left eigenfunction of $T$,
  we get:
  \[ \rho(T) \scal{h,f}  = \scal{h,Tf} \leq \scal{h,\rho(T)f} = \rho(T)\scal{h,f},\]
  where the inequality holds by positivity of $T$ and nonnegativity of $f$ and $h$. 
Therefore we have $\scal{h,Tf} = \scal{h,\rho(T)f}$, so $\scal{h,\rho(T)f - Tf} = 0$.
Since $\rho(T)f - Tf$ is nonnegative and $h>0 \muae$, this implies
$Tf=\rho(T)f$. 
\end{proof}

As a first consequence, we give details on which atoms may appear in the
support of  a nonnegative  eigenfunction.  Recall  that, for  a non-zero
atom $A$, the Perron like  eigenfunction $v_A$ is the right eigenfunction
of  $T_A$  given  by  Corollary~\ref{cor:rappel}.  For  $v\in  L^p_+$  a
nonnegative  eigenfunction of  $T$,  we consider  the following subset
of the atoms $\atom(\rho(v))$:
\[
  \atom_m(v) := \{ A \in \atom \, \colon\, A\subset
  \supp(v)\quad\text{and}\quad
   \rho(v)= \rho(A)\}. 
\]

  \begin{cor}[A dichotomy for atoms and nonnegative eigenfunctions]
    \label{cor:atoms_in_support}
Let $T$ be a positive power compact operator on $L^p$ with $p \in (1, +\infty)$.
Let  $v \in L^p_+$ be a nonnegative eigenfunction of $T$ with  $\lambda=\rho(v)>0$.
    \begin{enumerate}[(i)]
    \item
      \label{enum:dichotomy}
    For any atom $A$ with $A \subset \supp(v)$ a.e., 
    exactly one of the following holds:
    \begin{itemize}
      \item $\rho(A) < \lambda$ ; 
      \item   $\rho(A)  =   \lambda$,   that   is,  $A\in   \atom_m(v)$,
        $   \ind{A}  v   =   \cste  v_A$   for   some  $\cste>0$   and
        $\supp(v)\cap P^*(A) = \emptyset$ a.e..
      \end{itemize}
    \item\label{enum:source_of_f} The set of atoms
    $\atom_m(v)$ 
      is a nonempty finite antichain,  and:
 \[
\rho(v)=\rho(\supp(v)).
\]
\item If  $A\in \atom_m(v)$, $B\in \atom$  and $B\prec A$, then  we have
  $\rho(B)< \rho(A)$.
  \end{enumerate}
\end{cor}
\begin{proof}
  We start by proving~\ref{enum:dichotomy}.
  Let $v$, $\lambda$ satisfy the hypotheses, and consider an atom
  $A$ such that $A \subset \supp(v)$.
    If $\rho(A)<\lambda$ we are in the
    first case and there is nothing to prove. We now assume $\lambda \leq \rho(A)$.
    Since $T$ is a positive operator and $v$ is nonnegative, we have: 
    \begin{equation}
      \label{eq:subsol}
      T_A (v\ind{A})  = \ind{A} T (v\ind{A}) \leq \un_A T(v \un_A) + \un_A T( v \un_{A^c}) = \ind{A} Tv = \lambda \ind{A}v.
      \end{equation}
      Since $\lambda \leq \rho(A)$, and $A$ is
    irreducible, Proposition~\ref{prop:if_sub_then_eigen} applied to $T_{|A}$ implies 
    $T_{A} v_{A} = \rho(A) v_{A}$.
    Since  we  have $A  \subset  \supp(v)$,  $v_{A}$  is not  the  zero
    function,    thus,    by   Corollary~\ref{cor:rappel},    we    have
    $\lambda = \rho(A)$ and $\ind{A} v =\cste v_A$ for some $\cste>0$.
    Going back to \eqref{eq:subsol}, we see that the inequality there is
    in   fact  an   equality,   so  $\ind{A}   T(   v\ind{A^c})  =   0$.
    By~\eqref{eq:kT=0}, we thus have $k_T(A, \supp(v) \cap A^c) = 0$. By
    Lemma~\ref{lem:supp_inv}, the  set $\supp(v)$ is invariant,  thus by
    additivity of the kernel we also have:
    \[
k_T(A \cup \supp(v)^c , \supp(v)\cap A^c) = 0,
\]
    so that $\supp(v)\cap A^c$ is invariant. We then write
    \(
       F(\supp(v)\cap A^c) \cap A = (\supp(v) \cap A^c) \cap A = \emptyset,
    \)
    which implies by Lemma~\ref{lem:darknessoffuturepast} that:
    \begin{equation}
      \label{eq:notinthepast}
      \supp(v) \cap P^*(A) =\supp(v) \cap A^c \cap P(A) = 
      \emptyset.
    \end{equation}
This completes the proof of~Point~\ref{enum:dichotomy}

\medskip

We now turn to the proof of~\ref{enum:source_of_f}.  If two atoms $A$
and $B$ are in $\atom_m(v)$, Equation~\eqref{eq:notinthepast} shows
that $B$ cannot be a subset of~$P^*(A)$; symmetrically $A$ cannot be
included in $P^*(B)$.  By the alternate formulation of~$\preccurlyeq$
from Lemma~\ref{lem:def_rel_ordre}, $A$ and $B$ are not comparable, so
$\atom_m(v)$ is an antichain. It is finite by
Lemma~\ref{lem:prop-rho}~\ref{item:prop-rho_fini}.  Moreover, as  $T(v)
= \lambda v$, we get  that 
$T_{\supp(v)}(v) = \lambda v$, and thus 
$\rho(\supp(v)) \geq \lambda$.  As the set $\supp(v)$ is invariant by
Lemma~\ref{lem:supp_inv} (and thus admissible), by~\eqref{eq:rho=lax},
there exists an atom $A \subset \supp(v)$ with $\rho(A) \geq \lambda$,
and thus $\rho(A)=\lambda $ by Point~\ref{enum:dichotomy}. This implies
that  the finite antichain  $\atom_m(v)$ is not empty.

\medskip 

Finally,   if   $A\in\atom_m(v)$   and   $B\prec   A$,   then   we   get
$B\subset F(A)\subset \supp(v)$ since  $\supp(v)$ is invariant. Applying
the  dichotomy  from  Point~\ref{enum:dichotomy}, and  noting  that  $B$
cannot  be in  $\atom_m(v)$ since  it is  an antichain,  we deduce  that
$\rho(B)<\rho(A) = \lambda$.
\end{proof}

The last statement of Corollary~\ref{cor:atoms_in_support}
motivate the following definition, we refer to Figure~\ref{fig:dist}
for a pictorial representation. 
\begin{defi}[Distinguished atoms and eigenvalues]
  \label{def:distinguished}
  Let  $T$  be   a  positive  power  compact  operator   on  $L^p$  with
  $p  \in (1,  +\infty)$. A  non-zero atom  $A$ of  $T$ is  called right
  \emph{distinguished} if $\rho(B)< \rho(A)$ for  any atom $B$ such that
  $B\prec A$.

    The set of right distinguished atoms of radius $\lambda>0$ is denoted
    by $\distatom(\lambda)$.

    An   eigenvalue  $\lambda$   is   called   right  distinguished   if
    $\distatom(\lambda)\neq\emptyset$.
  \end{defi}
  One has a similar definition for left distinguished atoms/eigenvalues.
  When there  is no ambiguity,  we shall simply write  distinguished for
  right distinguished.  \medskip
  
By Corollary~\ref{cor:atoms_in_support}~\ref{enum:source_of_f}, if  $v$ is a
nonnegative eigenvalue, all atoms  in $\atom_m(v)$ are distinguished:
\begin{equation}
   \label{eq:umv-in-umd}
\atom_m(v)\subset\distatom.
\end{equation}
 In
the other direction, we now show that for any distinguished atom, we may
associate a nonnegative eigenfunction.  Recall that, for a non-zero atom
$A$,  $v_A$ denotes  the  Perron-like eigenfunction  of  $T_A$ given  by
Corollary~\ref{cor:rappel}.
  \begin{prop}[Nonnegative eigenfunctions associated to distinguished atoms]
    \label{prop:nonneg_eigen_dist_atom}
    Let $T$ be a positive power compact operator on $L^p$ with
    $p \in (1, +\infty)$, and $A$ a non-zero  atom.
    The following statements are equivalent:
\begin{enumerate}[(i)]
\item \label{prop:item:atom_dist}
  $A$ is a distinguished atom.
\item \label{prop:item:fstar_dist}
  $\rho(F^*(A))< \rho(A)$.
\item \label{prop:item:ex_vp_dist}
  There exists a nonnegative
  eigenfunction $w_A \in L^p_+$ such that $\supp(w_A) = F(A)$
 and $\un_A w_A  = v_A$.
\end{enumerate}
If they hold, then we have $\rho(w_A)=\rho(A)$.
\end{prop}

The condition $\un_A w_A  = v_A$
in~\ref{prop:item:ex_vp_dist} corresponds to a particular choice of
normalizing constant, see Lemma~\ref{lem:fct_supp_F(A)}. 

\begin{figure}
 \centering
\begin{tikzpicture}[
  common/.style={circle,draw},
  every child node/.style = common,
  dist/.style = {ultra  thick},
  level distance=12mm,
  baseline = (current bounding box.north),
  edge from parent/.style={draw,-latex}
  ]
  \node[common] {$1$}
  child {node[dist] {$3$}
    child {node[dist] {$2$}
      child {node[dist] {$1$}}
    }
    child {node {$1$}
      child {node[dist] {$1$}}
    }
    };
  \end{tikzpicture}%
  \hspace{2em}
  \begin{minipage}[t]{0.7\linewidth}
    Diagram  of the ordered set of  atoms.
    Following the classical convention (see~\cite[p.~4]{Bir67}),
    each circle represents an atom $A$, and is labeled with its radius
    $\rho(A)$. 
   An arrow  from atom $A$ to atom $B$ signifies
      that $B\prec A$ and there is no atom in between.
    
    The distinguished atoms are those circled in a thick line.

    Note that a family of  similar ``finite''  pictures may always be drawn
    in the general case, by considering only atoms with
    radius larger than a positive constant $\lambda$. 
  \end{minipage}
  \caption{Distinguished atoms%
    \label{fig:dist}%
}
\end{figure}

  \begin{proof}
Suppose that Point~\ref{prop:item:ex_vp_dist} holds,  and let $w_A$ be a
nonnegative eigenfunction with $\supp(w_A) = F(A)$. 
By Lemma~\ref{lem:fct_supp_F(A)}, we have $\rho(w_A) = \rho(A)$, so $A\in\atom_m(w_A)$,
and by~\eqref{eq:umv-in-umd},  it is distinguished.
Therefore Point~\ref{prop:item:ex_vp_dist} implies Point~\ref{prop:item:atom_dist}.

Suppose               that               Point~\ref{prop:item:atom_dist}
holds. By~\eqref{eq:rho=lax}, either $\rho(F^*(A)) = 0$, or there exists
an atom  $B\subset F^*(A)$  such that \(  \rho(F^*(A)) =  \rho(B)\).  By
Lemma~\ref{lem:def_rel_ordre}, this $B$ satisfies $B\prec A$.  Since $A$
is distinguished, $\rho(B)<\rho(A)$, so~Point~\ref{prop:item:fstar_dist}
holds.

\medskip

We now prove that Point~\ref{prop:item:fstar_dist} implies
Point~\ref{prop:item:ex_vp_dist}. Set $B=F^*(A)$. By assumption, the invariant set $B$
satisfies 
  $\rho(B) < \rho(A)$.  By
  Lemma~\ref{lem:inversepositive}, the operator $(\rho(A) \id - T_B)$
  is invertible and its inverse is a positive operator.  Let
  $w_A = v_A + f_B$, where
  \(
  f_B = (\rho(A) \id - T_B)^{-1}(\un_B T v_A)
  \).
  Note that, by the expression of $(\rho(A) \id - T_B)^{-1}$
  as a  Neumann series, we
  have $\supp(f_B)\subset B$, and thus  $\un_A w_A = v_A$.  Then we have:
 \begin{equation}
   \label{proof:eq:calcul_Tf_A}
T w_A =  T v_A + T f_B= \un_A T v_A + \un_{A^c} T v_A + T f_B. 
    \end{equation}
    As $\supp(f_B)$ is a subset of the invariant set~$B$,
    we know 
    by Lemma~\ref{lem:inv_sets_ideals} that $T f_B = T_B f_B$.
    Moreover, as $\supp(v_A) \subset A$, we have
    $\un_A T v_A = T_A v_A = \rho(A) v_A$ by definition of $v_A$.
    Finally, as the set $F(A)$ is invariant and as we have
    $\supp(v_A) \subset A \subset F(A)$, we have
    $\un_{F(A)^c} T v_A = 0$,  thus
    $\un_{A^c} T v_A = \un_B T v_A$.  Plugging this
    in~\eqref{proof:eq:calcul_Tf_A} yields:
 \begin{align*}
      Tw_A &= \rho(A) v_A + \un_B T v_A + \rho(A) f_B- \rho(A) f_B + T_B f_B\\
           &= \rho(A) w_A + \un_B T v_A - (\rho(A) \id - T_B)(f_B) \\
          &= \rho(A) w_A
 \end{align*}
    by definition of $f_B$. So $w_A$ is a nonnegative
    eigenfunction (with $\rho(w_A)=\rho(A)$). 
    In particular, $\supp(w_A)$ is an invariant set that contains
    $A$, so $F(A) \subset \supp(w_A)$. Since $\supp(v_A)$ and
    $\supp(f_B)\subset B$ are both subsets of $F(A)$, we get
    $F(A) = \supp(w_A)$. This proves Point~\ref{prop:item:ex_vp_dist}.
\end{proof}

The previous result shows that, to any distinguished $\lambda$, we may associate
a family $(w_A)_{A\in \distatom(\lambda)}$ composed of nonnegative
eigenfunctions. We now completely describe the set of  nonnegative eigenfunctions
associated to $\lambda$, say $V_+(\lambda)$,  as  the conical hull of
this family (that is linear combinations with nonnegative coefficients). 

\begin{theo}[Characterization of nonnegative right
  eigenfunctions]\label{th:carac_nonneg_vp_v2}
Let $T$ be a positive power compact operator on $L^p$ with $p \in (1,
+\infty)$. Let $\lambda > 0$.
We have the following properties. 
  \begin{enumerate}[(i)]
  \item\label{item:l-dist-atom}
    There exists a nonnegative eigenfunction of $T$ associated
    to $\lambda$ if and only if $\lambda$ is a distinguished eigenvalue. 
  \item\label{item:antichain_dist}
The set
$\distatom(\lambda)$ is a (possibly empty)  finite antichain of atoms, and the
family $(w_A)_{A\in \distatom(\lambda)}$ is linearly
    independent.
 \item\label{item:V+=conical_hull}
  If $v$ is a nonnegative eigenfunction with
  $\rho(v)=\lambda$, then $\lambda=\rho(\supp(v))$ and: 
    \[
      v = \sum_{A\in \atom_m(v)} \cste_A w_A
      \quad\text{with}\quad \cste_A>0.
    \]
So the cone $V_+(\lambda)$ is the conical hull of  $\{w_A\, \colon\, A\in
    \distatom(\lambda)\}$. 
  \end{enumerate}
\end{theo}
\begin{rem}
The last point shows in particular that if $w$ is a nonnegative
eigenfunction such that $\supp(w)=F(A)$, where $A$ is a non-zero atom
(see Lemmas~\ref{lem:supp_inv}
and~\ref{lem:fct_supp_F(A)}), then $A$ is distinguished,
$\rho(w)=\rho(\supp(w))=\rho(A)$ and $w=\cste w_A$ with $\cste>0$. 
\end{rem}

The  elementary   adaptation  of   the  theorem  to   nonnegative  left
eigenfunction is left to the reader.

\begin{proof}
  If $\lambda$ is distinguished, then by definition
  there is an atom $A\in\distatom(\lambda)$, and $w_A$
  provides a nonnegative eigenfunction associated to $\lambda$.
  Conversely, if there is a nonnegative eigenfunction $w$
  associated to $\lambda$, then $\atom_m(w)$ is nonempty
  and consists of distinguished atoms by Corollary~\ref{cor:atoms_in_support},
  so $\lambda$ is distinguished.
This proves Point~\ref{item:l-dist-atom}.

Let us prove Point~\ref{item:antichain_dist}.  If $A$ and $B$ belongs to
$\distatom(\lambda)$,  then  $\rho(A)  =   \rho(B)$,  so  they  are  not
comparable   by   definition    of   distinguished   atoms.    Therefore
$\distatom(\lambda)$  is  an antichain.  It  is  also finite  by  Lemma~
\ref{lem:prop-rho}~\ref{item:prop-rho_fini}.    To   prove  the   linear
independence              property,             assume              that
$\sum_{B\in  \distatom(\lambda)}  \cste_B  w_B  =  0$.  Multiplying  by
$\un_A$ for  $A\in\distatom(\lambda)$ yields  $\cste_A v_A =  0$, since
for $B\neq A$, $\supp(w_B) = F(B)$  is disjoint from $A$. Since $v_A$ is
positive, $\cste_A  = 0$. Since  this is true  for all $A$,  the family
$(w_A)_{A\in\distatom(\lambda)}$ is linearly independent.

\medskip

  We now prove Point~\ref{item:V+=conical_hull}.  Since the $w_A$ are
  all in the cone $V_+(\lambda)$, their conical hull is included in
  $V_+(\lambda)$, so that we only need to prove the reverse inclusion.
  Let $v\in V_+(\lambda)$.  By Corollary~\ref{cor:atoms_in_support},
  there is an antichain $\atom_m(v)\subset\distatom(\lambda)$ of distinguished atoms of radius
  $\lambda$ in the support of $w$, and  all other atoms in this support
  satisfy $\rho(B)<\lambda$. Define:
  \[
    B = \supp(v) \bigcap \bigg( \bigcup_{A\in \atom_m(v)} P(A) \bigg)^c
    = \supp(v)  \bigcap \bigg( \bigcup_{A\in \atom_m(v)} A \bigg) ^c,
  \]
  where the second equality follows from the fact that $\supp(v)\cap P^*(A) = \emptyset$ for all $A\in \atom_m(v)$,
  by Corollary~\ref{cor:atoms_in_support}. The first equality shows that $B$ is
  invariant. 

  Still following  Corollary~\ref{cor:atoms_in_support}, there exist
  $\cste_A>0$ such that $v\un_A = \cste_A v_A$ for $A\in \atom_m(v)$.
  Consider the function 
  $w=v - \sum_{A\in \atom_m(v)}\cste_A w_A$. Since $\supp(w_A) = F(A) \subset
  \supp(v)$, $\supp(w)$ is included in $\supp(v)$. 
  Since $w$  vanishes by construction on all atoms  $A\in\atom_m(v)$, we have
  in fact $\supp(w)\subset B$. Now, $T w = \lambda w$ since $v$
  and the $w_A$ are eigenfunctions.  Since $B$ is invariant and
  $\supp(w)\subset B$, we get that 
  $T_B w = \lambda w$. However, by construction,  $B$ cannot contain atoms of radius
  greater than or equal to $\lambda$, so $\rho(B) < \lambda$. Therefore
  $\lambda$ cannot be an eigenvalue of $T_B$, and 
  $w$ must be identically zero, so that $v = \sum_{A\in \atom_m(v)} \cste_A w_A$.
  Since $\atom_m(v)\subset \distatom(\lambda)$, we get that $v$ is in the conical
  hull of the $(w_A)_{A\in \distatom}$. This finishes the proof.
\end{proof}
  
\subsection{Monatomic operators: definition and characterization}
\label{sec:mono}

 In this section we shall consider positive power compact operators having only
 one non-zero atom,  which are  called \emph{monatomic}  operators ($T$
 is monatomic if $\card \anz=1$ with $\anz$ defined in~\eqref{eq:def-anz}).  We
 give in the  next theorem a characterization of  the monatomic positive
 power compact operators, see Theorem~\ref{thI:carac_monat}.

\begin{theo}[Characterization of monatomic operators]\label{th:carac_monat}
  Let   $T$   be   a   positive power compact operator   on   $L^p$   with
  $p \in (1, +\infty)$ such that $\rho(T)>0$.  The following properties are equivalent. 

\begin{enumerate}[(i)]
\item \label{th:item:T_mono_at}
  The operator $T$ is monatomic.
\item \label{th:item:ex_vp_gd_simple}
  There exist a unique right and a unique left
  nonnegative eigenfunctions of  $T$ with   non-zero eigenvalues,
  and $\rho(T)$ is a simple eigenvalue of $T$.
\item \label{th:item:ex_vp_gd_supp_non_vide}
  There exist a unique right and a unique left nonnegative
  eigenfunctions of $T$ with 
  non-zero   eigenvalues, say $u$ and $v$,   and
  $\supp       (u)       \cap        \supp(v)    $ has positive measure.
\end{enumerate}
Furthermore, when the operator
$T$ is monatomic, we have $\rho(u)=\rho(v)=\rho(T)$  and $\supp (u) \cap
\supp(v) $ is the non-zero atom of $T$. 
\end{theo}

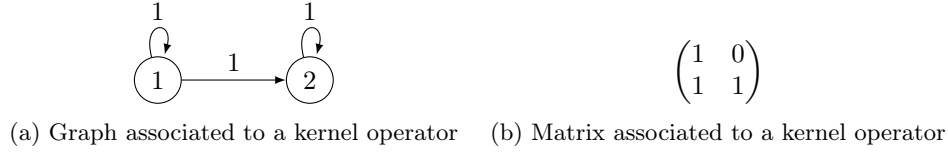
\begin{figure}
\centering
\begin{subfigure}[b]{.4\textwidth}
\centering
\begin{tikzpicture}
\node[draw,circle](1) at (0,0) {1};
\node[draw,circle](2) at (2,0) {2};
\draw[>=latex,->] (1)--(2) node[midway, above]{$1$};
\draw[>=latex,->] (1) edge[loop above] node[midway, above]{$1$} (1);
\draw[>=latex,->] (2) edge[loop above] node[midway, above]{$1$} (2);
\end{tikzpicture}
\caption{Graph associated to a kernel operator}
\label{fig:graph_supp1}
\end{subfigure}
\begin{subfigure}[b]{.4\textwidth}
\centering
$\begin{pmatrix}
1 & 0 \\
1 & 1 \\
\end{pmatrix}$
\caption{Matrix associated to a kernel operator}
\label{fig:graph_supp2}
\end{subfigure}
\caption{Example of associated graph and associated matrix of a kernel operator on $\Omega = \{1,2\}$}
\label{fig:graph_supp}
\end{figure}

\begin{ex}[On the condition $\rho(T)$ simple and $\supp(u)\cap \supp(v)
  $ with positive measure]
  If $T$ has a unique right and a unique left eigenfunction, then $T$ is
  not  monatomic in  general.   Indeed, consider  the  example given  by
  Fig.~\ref{fig:graph_supp}  with  $\Omega=\{1,  2\}$ endowed  with  the
  counting measure. The  positive kernel operator $T$  associated to the
  matrix  given   in  Fig.~\ref{fig:graph_supp2}  has  only   one  right
  eigenfunction $u = (0,1)$ and one  left eigenfunction $v = (1,0)$, but
  it is  not monatomic, as its  non-zero atoms are  $\{1\}$ and $\{2\}$.
  Here, we have  $\supp(u) \cap \supp(v) = \emptyset$ and  $\rho(T) = 1$
  is not a simple eigenvalue.
\end{ex}

To prove Theorem~\ref{th:carac_monat}, we use the following lemma.

\begin{lem}[Existence of minimal distinguished atoms]\label{lem:ex_dist}
Let $A$ be a non-zero atom. Then there exists a right (resp. left)
distinguished  atom smaller (resp. larger) than $A$ for $\preccurlyeq$,
say $B$, such that $\rho(B)\geq \rho(A)$. 
\end{lem}

\begin{proof}
  Recall that $T$ and $T^\star$ have the same spectral radius and that
  they share the same atoms, so we only need to prove the lemma for
  right distinguished atoms for $T$, as it will then hold for left distinguished
  atoms for $T$ as they are right distinguished atoms for  $T^\star$.

  Since $A$ is a non-zero atom, $\rho(A)$ is positive.  The set:
  \[
    \mathcal{A} = \{C\in \anz\,  \colon\, \rho(C)\geq  \rho(A), C \preccurlyeq A\}
  \]
  is finite thanks to Lemma~\ref{lem:prop-rho}~\ref{item:prop-rho_fini}
  and is non empty as it contains $A$.  Thus it has at least one
  minimal element for the order $\preccurlyeq$, say $B$.  If
  an atom $C$ satisfies
  $C\prec B$, then $C\preccurlyeq A$ by transitivity,
  but $C$ cannot be in $\mathcal{A}$ by minimality of $B$,  so $\rho(C)<\rho(A)$. Since
  $B\in\mathcal{A}$, we have $\rho(B)\geq \rho(A)$, and so $\rho(C)<\rho(B)$.
  Since this holds for any $C$ such that $C\prec B$, we obtain the atom $B$ is distinguished. 
  \end{proof}

\begin{proof}[Proof of Theorem~\ref{th:carac_monat}]
  We assume that $T$ is monatomic and prove
  Point~\ref{th:item:ex_vp_gd_simple}.  Let $A$ be the only non-zero
  atom.  By Lemma~\ref{lem:prop-rho}~\ref{item:prop-atom(rho)}, as
  $\mult(\rho(T), T)\geq 1$ and $\anz$ is reduced to $\{A\}$, we get
  that 
  $\rho(T)$ is simple and $\rho(A)=\rho(T) $ by~\eqref{eq:rho=lax}.

  We now prove the existence and uniqueness of a nonnegative
  right eigenfunction.
  Since there is no other non-zero atom,
using directly Definition~\ref{def:distinguished} we see that
$A$ is distinguished, and is the only distinguished atom.
Still by definition, $\rho(A)$ is the only distinguished eigenvalue.
By Theorem~\ref{th:carac_nonneg_vp_v2}, 
  the set of  nonnegative eigenfunctions is the cone
  $\mathbb{R}_+ w_A$,
  which proves uniqueness (up to a positive multiplicative constant). 
  Applying the same proof to $T^\star$ gives
  Point~\ref{th:item:ex_vp_gd_simple} and the first part of the last
  sentence of the theorem.

  \medskip

  We assume Point~\ref{th:item:ex_vp_gd_simple} and prove
  Point~\ref{th:item:ex_vp_gd_supp_non_vide}.  Since $\rho(T)>0$ is
  simple, we deduce from~\eqref{eq:mult} that there exists a unique
  atom, say $A$, such that $\rho(A)=\rho(T)$.  In particular, all
  other atoms must satisfy $\rho(B)<\rho(A)$, so that $A$ is right
  (and left) distinguished.  Therefore, by
  Proposition~\ref{prop:nonneg_eigen_dist_atom}, the unique right
  (resp. left) nonnegative eigenfunction, whose existence is given by
  our Assumption, is in fact $w_A$ (resp. the nonnegative
  eigenfunction $w_A^\star$ obtained from $T^\star$).  Since
  $\supp(w_A)\cap \supp(w_A^\star) = F(A) \cap P(A) = A$ by convexity of
  the atom $A$, we obtain Point~\ref{th:item:ex_vp_gd_supp_non_vide}
  and the last part of the last sentence of the theorem.

\medskip

We assume Point~\ref{th:item:ex_vp_gd_supp_non_vide}  and prove that the
operator $T$ is monatomic. Since  $\rho(T)>0$, there exists an atom, say
$A$,  such that  $\rho(A)=\rho(T)$.  Looking for a  contradiction, we  assume there
exists an other non-zero atom  $B$ and without loss of generality  that it is not
smaller than $B$ for $\preccurlyeq$ (that is, either $A\preccurlyeq B$ or
$A$ and $B$  are not comparable), equivalently  $ F(A)\cap B=\emptyset$.
By  Lemma~\ref{lem:darknessoffuturepast}, this  is also  equivalent to
$F(A) \cap P(B)=\emptyset$.

 Then,  using Lemma~\ref{lem:ex_dist},
there exists a  right (resp. left) distinguished atom  $A'$ (resp. $B'$)
such that $A'  \preccurlyeq A$ (resp.  $B \preccurlyeq  B'$).
By Proposition~\ref{prop:nonneg_eigen_dist_atom},   the    unique   non    negative   right
eigenfunction  $v$  must  satisfy $\supp(v)=F(A')$,  and  similarly  the
unique   non    negative   left    eigenfunction   $u$    must   satisfy
$\supp(u)=P(B')$.   By construction,  we  have  $F(A')\subset F(A)$  and
$P(B')            \subset           P(B)$,            and           thus
$\supp(v)\cap \supp(u)= F(A') \cap P(B') \subset F(A) \cap
P(B)=\emptyset$. As this is in contradiction with the assumption of
Point~\ref{th:item:ex_vp_gd_supp_non_vide}, we deduce that $A$ is the
only non-zero atom, that is $T$ is monatomic. 
\end{proof}

\section{Generalized eigenspace at the spectral radius}\label{sec:gen_eigenspace}

\subsection{Framework and  main theorem}

The  purpose   of  this   section  is   to  restate   \cite[Theorem  V.1
(2)]{janglewisvictory}  on  the  ascent  of  $T$  in  our  framework  of
$L^p$-spaces, with a shorter proof based on convex sets.

\medskip

Let us first recall a few classical definitions, see~\cite{dunford88} and~\cite{konig86}.
For $T$ an bounded operator on a Banach space and 
$\lambda \in \C$, we call \emph{generalized eigenspace} of $T$ at
$\lambda$, and denote by  $K(\lambda, T)$, the linear subspace:
\[
  K(\lambda, T) = \bigcup\limits_{k \in \N} \ker (T - \lambda \id)^k.
\]
We  now focus  on the  spectral radius  $\lambda =  \rho(T)$, and  write
$K(T)  =  K(\rho(T),T)$  the corresponding  generalized  eigenspace.  We
define the  \emph{index} of a  generalized eigenvector $u \in  K(T)$, as
$\inf\{k \in  \N \, \colon\,  u \in  \ker(T -  \rho(T) \id)^k  \}$, and,  with the
convention $\inf \emptyset=+\infty$, the \emph{ascent} of $T$ at  $\rho(T)$ as:
\[
   \alpha_T = \inf  \{ k  \in \N :  \ker (T - \rho(T)  \id)^k
   =\ker (T - \rho(T) \id)^{k+1}\}.
 \]
   Notice that $\alpha_T$ is positive if $\rho(T)$ is an eigenvalue and
   that 
$K(T) = \ker(T - \rho(T) \id)^{\alpha_T}$  when $\alpha_T$
is finite. When the operator $T$ is power compact,
then the ascent $\alpha_T$ is finite, see
\cite[Lemma 1.a.2, Theorem p.~21]{konig86} (it is also equal to the 
 descent $\delta_T = \inf  \{ k  \in \N :  \im (T - \rho(T)  \id)^k
   =\im (T - \rho(T) \id)^{k+1}\}$). 
\medskip

Let  $T$   be  a   positive  power  compact   operator  on   $L^p$  with
$p   \in    (1,   +\infty)$,   and   assume    $\rho(T)>0$,   and   thus
$\alpha_T\in                          \N^*$.                          By
Lemma~\ref{lem:prop-rho}~\ref{item:prop-atom(rho)},  $K(T)$   is  finite
dimensional, and:
\[
  \dim (K(T)) = m(\rho(T),T) = \card ( \critatom) ,
\]
where $\critatom$ is the set of
\emph{critical atoms}:
\begin{equation}
  \label{eq:def-ucrit}
    \critatom = \{ A\in \atom\, \colon\,  \rho(A) = \rho(T)\}.
\end{equation}
By definition of~$\alpha_T$, the  sequence
\(
  (\dim( \ker( (T-\rho(T) \id)^k)))_{1\leq k \leq \alpha_T}
\)
is (strictly) increasing, so we have the following trivial bounds:
\begin{equation}
  \label{eq:trivial_bound_ascent}
  \dim\left(\ker(T-\rho(T) \id)^k\right) \geq k,
\quad\text{for all}\quad
 1\leq k \leq \alpha_T,
\end{equation}
and in particular $\dim(K(T)) = \card(\critatom) \geq \alpha_T$. 

The set $\critatom$ may be equipped with the partial order $\preccurlyeq$.
Recall that we write $B\prec A$ if $B\preccurlyeq A$ and $B\neq A$. 
We recall a few classical definitions for posets, that is, partially
ordered sets (see e.g. 
\cite[Section I.3, p.~4]{Bir67}). 
\begin{defi}[Covering]
  Let $A$ and $B$ be critical atoms. If $B\prec A$, and if there is no critical
  atom~$C$ such that $B\prec C \prec A$, then $A$ is said to \emph{cover} $B$. 
  \end{defi}

  For $n\geq 1$, a  \emph{chain of length $n$} is a sequence
  $(A_0, \dots, A_n)$ of elements of $ \critatom$  such that
  $ A_{i+1}\prec A_i$ for all $0\leq i < n$.  The \emph{height} $h(A)$ of
  a critical  atom $A$,
  is one plus  the maximum length of a chain starting at $A$.
  
\begin{rem}[Terminology - off by one]
  Our definition of length is consistent with~\cite[Section I.3]{Bir67}.
  The ``off by one'' is due to the fact that height, in \cite{Bir67}, is
  formally defined for posets  with a least element. Our height
  coincides   with    Birkhoff's   height   on   the    poset
  $(\critatom \sqcup\{  \mathbf{0}\}, \preccurlyeq)$  where $\mathbf{0}$
  is an  additional element  that satisfies $\mathbf{0}  \preccurlyeq A$
  for all $A\in\critatom$.
\end{rem}
  
We  now restate \cite[Theorem~V.1~(1, 2)]{janglewisvictory} in our
framework; its proof is given in Section~\ref{sec:proof-wA}. Recall
$v_A$ the Perron-like eigenfunction of $T_A$ and the set of critical
atoms $\critatom$ from~\eqref{eq:def-ucrit}.

\begin{theo}[A basis of $K(T)$]%
  \label{th:base_K}
  Let  $T$  be   a  positive  power  compact  operator   on  $L^p$  with
  $p \in (1, +\infty)$ with a spectral radius $\rho(T) > 0$.  Then there
  exists  a basis  $\mathcal{W} =  (w_A)_{ A  \in \critatom}$  of $K(T)$
  satisfying the following properties:
\begin{enumerate}[(i)]
\item\label{it:supp_wA}
  For all $A$,  $A\subset \supp(w_A) \subset F(A)$, and $\un_A w_A = v_A$ ;
  moreover  if $A$ is distinguished then $w_A$ is the
   nonnegative eigenfunction introduced in  Proposition~\ref{prop:nonneg_eigen_dist_atom}. 
 \item\label{it:cover} If $M = (M_{A,B})$ is the matrix representing,
   on the basis $\mathcal{W}$, the endomorphism induced on $K(T)$ by
   $T$, then for $A, B\in \critatom$, we have:
  \[
M_{AB}=    \begin{cases}
       0 &\text{ if } B \npreccurlyeq A, \\
       \rho(T)&\text{ if } A=B , \\
       >0 & \text{ if } A \text{ covers } B.
    \end{cases}
  \]
\item\label{it:index_wA}
  For any $A \in \critatom$, the index of $w_A$ is the height $h(A)$.
\end{enumerate}
  Moreover, Properties \ref{it:cover} and~\ref{it:index_wA} hold for 
  any basis of $K(T)$  satisfying~\ref{it:supp_wA}. 
\end{theo}

Since the ascent is the maximum index of functions  in $K(T)$,
we easily get the following result. 

\begin{cor}[Ascent and maximal height]
  \label{cor:ascent}
  The ascent of $T$ at its spectral radius $\rho(T)$ is equal to the
  maximal height of the critical atoms: 
\[
  \alpha_T = \max_{A \in \critatom} h(A).
\]
\end{cor}

\subsection{Existence of an adapted basis and  proof of Theorem~\ref{th:base_K}}
\label{sec:proof-wA}

We first state a key technical result.

\begin{lem}[Generalized eigenspaces for restrictions]
  \label{lem:rest_generalized_eigenspace}
  Let  $A$ be a convex set and $\lambda\in \C$.
  \begin{enumerate}[label=(\roman*)]
    \item\label{it:rest_cvx} If $v\in K(\lambda,T)$ and  $\supp(v)
      \subset F(A)$, then we have  $(\un_{A}\, v) \in K(\lambda,T_A)$.
\item \label{it:rest_converse}
If furthermore $A$ is invariant, and $\lambda\neq 0$,  then we have
  $K(\lambda,T_A) \subset K(\lambda,T)$. 
\end{enumerate}

\end{lem}
\begin{proof}
  If $\supp(v)\subset F(A)$, then by Lemma~\ref{lem:T_equal_TC}, we have
  $\un_A
  T v = T_A (\un_A\, v)$. 
  An easy induction using the identity $(T^j)_A = (T_A)^j$ from
  Lemma~\ref{lem:pow_rest} yields $\un_A T^j v = T_A^j (\un_A\, v)$ for all $j\geq 1$,
  and since this still holds for $j=0$, we get: 
  \begin{equation}
    \label{eq:brazouf}
    \un_A (T-\lambda \id)^j v = (T_A - \lambda \id)^j (\un_A\, v). 
  \end{equation}
  This proves the first item.

  If $(T_A - \lambda \id)^k v = 0$, the expression $(-\lambda)^k v = - \sum_{j=1}^k \binom{k}{j} (-\lambda)^{k-j}T_A^j v$
  shows that $\supp(v) \subset A$. By invariance this implies
  $\supp(T^j v) \subset A$, so $(T-\lambda \id)^k v = \un_A (T-\lambda \id)^k v$.
  We may now apply~\eqref{eq:brazouf}, as invariant sets are convex, and
  get  $ \un_A (T-\lambda \id)^k v =  (T_A - \lambda \id)^k v = 0$,
  which concludes the proof. 
\end{proof}

\begin{cor}%
  \label{cor:gen_ev_near_greatest_atom}
  Let $A\in\critatom$, $B=\bigcup_{C\in \critatom, C\prec A} C$, and  $\subA = F(A)
  \setminus F(B)$. 
  \begin{enumerate}[(i)]
  \item The set $\subA$ contains $A$, it is convex, $ F^*(\subA)= F(B)$ and
    $F(A) = \subA \sqcup F^*(\subA)$. 
  \item \label{cor:item:w_tilde_A}
    There exists a nonnegative eigenfunction $w_{\subA}$ of $T_{\subA}$
    such that 
    $\supp(w_{\subA}) = \subA$,   $\un_A\, w_{\subA} = v_A$, and
    $\rho(w_{\subA})=\rho(T)$. 
  \item \label{cor:item:w_on_tilde_A}
    If $w\in K(T)$ satisfies $\supp(w) \subset F(A)$, then  
  there exists $\cste \in \mathbb{R}$ such that
  \(\un_{\subA}  w  = \cste w_{\subA}\).
\end{enumerate}
\end{cor}
\begin{proof}
  The  set $\subA$  is  convex,  since it  is  the  intersection of  the
  invariant set $F(A)$  with the co-invariant set $F(B)^c$.  The set $A$
  cannot intersect $F(B)$, since this would imply $A\prec A$, so $\subA$
  contains  $A$. By  definition  of $F(B)$,  $\subA$  contains no  other
  critical atoms.  Therefore $A$ is distinguished for $T_{\subA}$, which
  yields       the        existence       of        $w_{\subA}$       by
  Proposition~\ref{prop:nonneg_eigen_dist_atom};                moreover
  $K(\rho(T), T_{\subA}) = \vect( w_{\subA})$ as $\rho(T)$ is simple for
  $T_{\subA}$.                                                        By
  Lemma~\ref{lem:rest_generalized_eigenspace}~\ref{it:rest_cvx},     the
  function $\un_{\subA} \,  w $ belongs to $  K(\rho(T), T_{\subA})$ and
  is therefore proportional to $w_{\subA}$, as claimed.
\end{proof}

We are now in a position to prove Theorem~\ref{th:base_K}. We proceed in
several steps.  

\subsubsection{Existence of a basis satisfying~\ref{it:supp_wA}}
We prove the existence of a basis satisfying Theorem~\ref{th:base_K}~\ref{it:supp_wA} by induction
  on the number of critical atoms of $T$.

  If $T$ has one critical atom $A$, then $A$ is necessarily distinguished. The nonnegative
  eigenfunction $w_A$ given by Proposition~\ref{prop:nonneg_eigen_dist_atom} is a non-zero vector in the one-dimensional vector
  space $K(T)$, so it is indeed a basis.

  For the induction step, assume that for any positive power compact operator $U$ on $L^p$ with
  at most~$n$ critical atoms, there exists a basis of $K(U)$ satisfying \ref{it:supp_wA}.
  Let $T$ be a positive power compact operator on $L^p$ with $n+1$ critical atoms.

  We first claim that, for each critical atom $A$ of $T$, there exists
  $w_A \in K(T)$ such that $A\subset \supp(w_A) \subset F(A)$.
  Indeed, there are two cases. If $T_{F(A)}$ has $n$ atoms or less,
  then the induction hypothesis applied to $U = T_{F(A)}$ gives the
  existence of $w_A\in K(U)$ such that
  $A\subset \supp(w_A) \subset F(A)$, $\un_A \, w_A= v_A$, and by
  Lemma~\ref{lem:rest_generalized_eigenspace}~\ref{it:rest_converse},
  $w_A$ is in fact in $K(T)$, proving the claim in this case.  If
  $T_{F(A)}$ has $n+1$ atoms, then all critical atoms of $T$ are in
  the future of $A$.  Notice that $\rho(T_{F(A)})=\rho(F(A))=\rho(T)$
  and by
  Lemma~\ref{lem:rest_generalized_eigenspace}~\ref{it:rest_converse}
  $K(T_{F(A)})\subset K(T)$. Furthermore, all the critical atoms of
  $T$ belongs to $F(A)$ and are thus the critical atoms of $T_{F(A)}$;
  this implies that
  $\dim(K(T_{F(A)}))= \card(\critatom) = \dim(K(T))$. We deduce that
  $K(T_{F(A)})= K(T)$.  Let $\subA$ be defined by
  Corollary~\ref{cor:gen_ev_near_greatest_atom}, and let
  $U = T_{F^*(\subA)}$.  Let $w\in K(T)=K(T_{F(A)})$.  We thus have
  $\supp(w) \subset F(A)$.  By
  Corollary~\ref{cor:gen_ev_near_greatest_atom}~\ref{cor:item:w_tilde_A}-\ref{cor:item:w_on_tilde_A},
  if $w$ vanishes on $A$, then it must be identically zero on $\subA$.
  Therefore we get $\supp(w) \subset F^*(\subA)$ and $w\in K(U)$ by
  Lemma~\ref{lem:rest_generalized_eigenspace}~\ref{it:rest_cvx}, since
  $F^*(\subA)$ is convex.
 As    a    consequence,   since    by
  Lemma~\ref{lem:prop-rho}~\ref{item:prop-atom(rho)},
  $\dim(K(T)) = n+1 > n = \dim(K(U))$, at least one element of $K(T)$ is
  non-zero             on              $A$.              By
  Corollary~\ref{cor:gen_ev_near_greatest_atom}~\ref{cor:item:w_on_tilde_A}
  we     may    assume     without    loss     of    generality     that
  $\un_{\subA}\, w=w_{\subA}$. In particular, $\un_A\, w = v_A$, and the
  claim is proved.
  
    Now, a family  $\mathcal{W} = (w_A)_{A\in\critatom}$ satisfying the claim  must be linearly independent. Indeed, assume that
    $\sum_{A\in\critatom} \cste_A w_A = 0$. If the $\cste_A$ do not vanish, let $B$ be a maximal element (for
    $\preccurlyeq$) among the atoms for which $\cste_B\neq 0$. For any atom
    $A\neq B$, either $ B\npreccurlyeq A  $ and $w_A$ is zero on $B$,
    or $B \prec A $ and $\cste_A = 0$ by maximality of $B$.
    Therefore $0 = 0\un_B = (\sum_A \cste_A w_A) \un_{B} =  \cste_Bw_B\un_B$,
    so $\cste_B = 0$, a contradiction. Therefore all $\cste_A$ must vanish,
    and the family~$\mathcal{W}$ is linearly independent.

    This independence and the fact that $\card(\critatom) = \dim(K(T))$ ensure
    that $\mathcal{W}$ is a basis: this completes the induction and
    proves Point~\ref{it:supp_wA}. 

\subsubsection{Proof of~\ref{it:cover}: the two-atoms case}    
We first prove Theorem~\ref{th:base_K}~\ref{it:cover} under the
additional assumption that $T$ has only two critical atoms $A$ and
$B$, and that $B\prec A$.
    
    By the trivial bound~\eqref{eq:trivial_bound_ascent}, the ascent is either equal to $1$,
    in which case $\ker(T-\rho(T)) = K(T)$ is two-dimensional, or
    equal to $2$, in which case
    $1  = \dim(\ker(T-\lambda \id)) < \dim(\ker((T-\lambda \id)^2)) = \dim(K(T)) = 2$.
    Let $(w_A,w_B)$ be a basis of $K(T)$ given by Point~\ref{it:supp_wA}.

    Note that $K(T)$ is stable by $T$, so there exist four coefficients such that:
    \begin{align*}
      Tw_A &= M_{AA} w_A + M_{AB} w_B, \\
      T w_B &= M_{BA}w_A + M_{BB} w_B.
    \end{align*}

    Since $B$ is distinguished, $w_B$ is the nonnegative eigenvector
    given by Proposition~\ref{prop:nonneg_eigen_dist_atom}, so  $M_{BB} = \rho(T)$ and  $M_{BA} = 0$.

    The support of $w_A$ is included in the future of the convex
    set $A$, so by Lemma~\ref{lem:T_equal_TC} we get
    $T_A (\un_A w_A) = T_A(w_A) = \un_A Tw_A
    = M_{AA} \un_A\, w_A $, since $w_B = 0$ on $A$.
    Since $w_A = v_A$ on $A$, we see that $M_{AA} = \rho(T)$.
    We may therefore write:
\begin{equation}
      \label{eq:what_is_MAB}
      (T-\rho(T) \id) w_A = M_{AB} w_B,
    \end{equation}
    and establishing Theorem~\ref{th:base_K}~\ref{it:cover} in this
    case consists in proving that $M_{AB}$ is positive.  Let $v_B^\star$
    be a positive Perron eigenvector of $T_B^\star$. Since the future of
    $B$ for $T^\star$ is $P(B)$, we have:
    \[
      T^\star  v_B^\star =  T^\star_B v_B^\star +  \un_{P^*(B)} T^\star
      v_B^\star = \rho(T) v_B^\star  + \un_{P^*(B)} T^\star v_B^\star.
    \]
    Taking the scalar product with $v_B^\star$ in \eqref{eq:what_is_MAB}
    yields:
    \begin{align*}
      M_{AB} \scal{v_B^\star, w_B} &= \scal{v_B^\star, (T-\rho(T)) w_A}\\
                               &= \scal{ T^\star v_B^\star - \rho(T) v_B^\star,w_A} \\
                               &= \scal{\un_{P^*(B)} T^\star v_B^\star, w_A} \\
      &= \scal{v_B^\star, T(  \un_{P^*(B)}\, w_A)}.
    \end{align*}
    
    By~Corollary~\ref{cor:gen_ev_near_greatest_atom},
    $\un_{P^*(B)}\, w_A$ is nonnegative, and positive on
    $\subA = F(A)\setminus F(B)$, so the last expression is
    nonnegative. Since the scalar product $\scal{v_B^\star, w_B}$ is
    positive, $M_{AB}$ is nonnegative.  Assume for a moment that
    $M_{AB} =0$, so that $\scal{v_B^\star, T(w_A \un_{P^*(B)})} = 0$, and
    by \eqref{eq:kT=0}, $k_T(B,\subA) = 0$. Using the partition
    $\Omega = F(A)^c \sqcup \subA \sqcup B \sqcup F^*(B)$ and the invariance
    of $F(A)$, we easily check that
    $k_T(B\cup F(A)^c, \subA\cup F^*(B))= 0$, so $\subA\cup F^*(B)$ is
    invariant. Since it contains $A$, it must contain $F(A)$, and
    therefore $B$, a contradiction.  This shows that $M_{AB}>0$,
    concluding the proof of the two-atoms case.  Note that
    $M_{AB}\neq 0$ also shows that $w_A \notin \ker(T-\rho(T) \id)$, so
    that the ascent is necessarily equal to two.

\subsubsection{Proof of~\ref{it:cover}: general case}
    By definition, for all $A$, we have:
    \begin{equation}
      \label{eq:matrixOfT}
      T w_A = \sum_{B\in \critatom} M_{AB} w_B = \sum_{B\in\critatom, B\prec A} M_{AB} w_B + M_{AA}w_A +
      \sum_{B\in\critatom, B\npreccurlyeq A} M_{AB} w_B.
    \end{equation}
    Since $\supp(w_A)\subset F(A)$, we have $w_A \in K\left(\rho(T),T_{F(A)}\right)$,
    so Point~\ref{it:supp_wA} applied
    to $T_{F(A)}$ shows that $M_{AB} = 0$ if $B\npreccurlyeq A$.  
    Then, multiplying~\eqref{eq:matrixOfT} by $\un_A$ and applying
    Corollary~\ref{cor:gen_ev_near_greatest_atom}  yields
    $\rho(T)v_A = M_{AA}v_A$, so $M_{AA} = \rho(T)$.

    Assume now  that $A$  covers $B_0$,  and let $C$  be the  convex set
    $F(A) \cap  P(B_0)$: by definition,  the only critical atoms  in $C$
    are $A$  and $B_0$. For any  other atom $B$, either  $B\nprec A$ and
    $M_{AB}   =   0$,   or   $B\prec    A$   but   $B_0\nprec   B$,   so
    $F(B) \cap  C = \emptyset$,  and $w_B$  is zero on  $C$.  Therefore,
    multiplying by $\un_{C}$ in~\eqref{eq:matrixOfT} yields:
    \[
      \un_{C} \, T w_A = \rho(T) \un_C\, w_A  + M_{AB_0}  \un_{C}\, w_{B_0}.
    \]
    Using Lemma~\ref{lem:T_equal_TC}, and the fact that
    $\un_C\, w_{B_0}  = \un_{B_0} \, w_{B_0} = v_{B_0}$, we get
    $T_C  (\un_C w_A) = \rho(T) (\un_C w_A) + M_{AB_0} v_{B_0}$, so
    $M_{AB_0}$ is a term of the matrix of $T_C$ in the basis
    $(\un_{C} w_A,v_{B_0})$ of $K(T_C,\rho(T_C))$, and
    its positivity follows from the two-atoms case. 

    \subsubsection{Conclusion}

    To check that Point~\ref{it:index_wA} of Theorem~\ref{th:base_K} holds,
    note that the matrix $N$ of $S = T-\rho(T) \id$ on the basis $\mathcal{W}$ satisfies
    $N_{AB} = 0$ unless $B\prec A$, and $N_{AB}>0$ if $A$ covers
    $B$. Thus, we get:
    \[
      (N^k)_{AB} = \sum_{A = A_0 \succ A_1\cdots \succ A_k = B} \prod_j
      N_{A_j,A_{j+1}}.
    \]
    If $k>h(A)$, there is no chain of length $k$ starting down from $A$, so
    $N^k w_A = 0$. If $k=h(A)$, the sum is non-empty, the only chains appearing in the sum
    are of maximal length so $A_j$ must cover $A_{j+1}$, the corresponding products are all positive, so $N^kw_A = \sum_B \cste_B w_B$
    for some non-zero numbers $\cste_B$, and $N^k w_A \neq 0$. Therefore
    the index of $w_A$ is $h(A)$.
  
    Notice the proof  of Points~\ref{it:cover} and~\ref{it:index_wA} are
    done    under   the    condition    that    the   basis  only  satisfies
    Point~\ref{it:supp_wA}. This completes the proof of Theorem~\ref{th:base_K}. 

\bibliographystyle{abbrv}
\bibliography{biblio}

\end{document}